\documentclass[a4paper,fleqn]{cas-dc}
\usepackage[numbers,sort&compress]{natbib}
\usepackage{amsmath,amsthm,amssymb,amsfonts}
\usepackage{graphicx}
\usepackage{textcomp}
\usepackage{xcolor}
\usepackage{mathrsfs}
\usepackage{indentfirst}
\usepackage{multirow}
\usepackage{float}
\usepackage{etoolbox}
\allowdisplaybreaks[4]
\newtheorem{theorem}{Theorem}
\newtheorem{lemma}{Lemma}
\newtheorem{corollary}{Corollary}

\newtheorem{assumption}{Assumption}

\newtheorem{example}{Example}
\newtheorem{remark}{Remark}
\newtheorem{problem}{Problem}

\def\BibTeX{{\rm B\kern-.05em{\sc i\kern-.025em b}\kern-.08em
    T\kern-.1667em\lower.7ex\hbox{E}\kern-.125emX}}

\ExplSyntaxOn
\cs_set:Npn \__first_footerline: { }
\ExplSyntaxOff

\begin{document}
\let\WriteBookmarks\relax
\def\floatpagepagefraction{1}
\def\textpagefraction{.001}

\shorttitle{Fully Coupled Nonlinear Mean-Field FBS$\Delta$Es}
\shortauthors{Zheng et al.}
\title[mode=title]{Fully Coupled Nonlinear Mean-Field FBS$\Delta$Es: Solvability and an LQ Buffer-Adjustment Illustration}

\author[1]{Yingjie Zheng}
\author[1]{Qingxin Meng}
\author[1]{Maoning Tang}\cormark[1]
\cortext[cor1]{Corresponding author.}
\ead{tmorning@huznu.edu.cn}

\address[1]{Department of Mathematical Sciences, Huzhou Normal University, Huzhou 313000, Zhejiang, China}

\fntext[fn1]{This work was supported by the National Natural Science Foundation of China (No.~12271158), the Key Program of Zhejiang Provincial Natural Science Foundation (No.~LZ22A010005), the Natural Science Foundation of Huzhou City (No.~2023YZ46), and the Postgraduate Research and Innovation Project of Huzhou Normal University (No.~2024KYCX66).}

\begin{abstract}
This paper establishes sufficient conditions for finite-horizon solvability of
fully coupled nonlinear mean-field forward--backward stochastic difference
equations with dependence on unconditional first moments. The forward recursion
uses conditional projections of the next backward state and its product with
the innovation. Building on existing domination--monotonicity methods, we
formulate deterministic matrix combinations in centered and mean coordinates
and prove a continuation estimate uniform in the homotopy parameter. Global
Lipschitz continuity and one active domination--coercivity direction yield a
unique square-integrable adapted solution, an a priori bound, and coefficient
stability. The active parameter can be normalized without imposing an additional
smallness restriction on the original coefficients. A sign transformation
handles the opposite monotonicity orientation. A nonlinear example with
saturating state and mean interactions verifies the assumptions, including
degenerate domination directions. A scalar mean-field LQ buffer-adjustment
model illustrates the theorem: its Hamiltonian solution characterizes the
unique open-loop optimizer. An exact finite scenario-tree calculation checks
this characterization against direct quadratic optimization.
\end{abstract}

% Highlights page removed (the CAS template typesets it as a separate first page).
% The text is kept here commented out in case it is needed for a journal submission.
% \begin{highlights}
% \item Centered and mean matrix combinations support discrete-time continuation
% \item Solvability and stability hold uniformly in the homotopy parameter
% \item A nonlinear example verifies the assumptions with degenerate directions
% \item An LQ buffer model is checked by exact finite scenario-tree optimization
% \end{highlights}


\maketitle

\section{Introduction}
Forward--backward stochastic systems provide a natural
framework for stochastic control, differential games, mathematical finance,
and recursive optimization. Since the seminal work on backward stochastic
differential equations~\cite{Pardoux1990}, the well-posedness of coupled
forward--backward stochastic differential equations (FBSDEs) has been
extensively studied through monotonicity conditions, continuation methods,
and decoupling techniques; see, among others,
\cite{Antonelli1993,Hu1995,Ma1999,Delarue2002,Yong1997,Yong2010,Ma2015}.
Their discrete-time counterparts, forward--backward stochastic difference
equations (FBS$\Delta$Es), arise naturally from stochastic maximum
principles and Hamiltonian systems for discrete-time stochastic control.
Linear FBS$\Delta$Es and their control applications have been studied in
\cite{Xu2017,Xu2018,Song2020}. Beyond the linear setting, Ji and
Liu~\cite{JiLiu2024} investigated solvability for a class of
forward--backward stochastic difference equations, while
Song~\cite{Song2023} considered general fully coupled FBS$\Delta$Es with
delay and related control applications. On the control-theoretic side,
Dong, Nie, and Wu~\cite{DongNieWu2022} developed a maximum principle for
discrete-time mean-field stochastic control problems, and Song and
Wu~\cite{SongWu2024} obtained a general maximum principle for discrete-time
stochastic control problems with possibly non-convex control domains.
These works characterize optimality conditions for controlled discrete-time
systems driven by stochastic difference equations, rather than a
coefficient-level solvability theory for the underlying uncontrolled
mean-field FBS$\Delta$E. Numerical and approximation aspects
connecting continuous- and discrete-time backward systems can be found in
\cite{Ma2002,Gobet2005,Zhang2004,Bender2008,Delarue2006,Cheridito2013}.

A distinctive feature of discrete-time FBS$\Delta$Es is that the backward
component enters the forward dynamics through conditional projections.
For example, the discrete duality calculation gives
\begin{equation*}
	\mathbb E\langle \sigma_k w_k,y_{k+1}\rangle
	=
	\mathbb E\langle \sigma_k,z'_{k+1}\rangle,
	\quad
	z'_{k+1}
	=
	\mathbb E[y_{k+1}w_k\mid\mathcal F_{k-1}],
\end{equation*}
rather than the continuous-time It\^o pairing between the diffusion
coefficient and a martingale-integrand process. Consequently, solvability
arguments for fully coupled discrete-time systems must explicitly account
for the projected variables $(y'_{k+1},z'_{k+1})$. This difficulty is
already present in the non-mean-field setting. Recently, Niu et
al.~\cite{Niu2026a} developed a domination--monotonicity framework for
fully coupled nonlinear FBS$\Delta$Es, providing a discrete-time analogue
of the domination--monotonicity approach initiated by Yu~\cite{Yu2022} and
further developed for mean-field type FBSDEs by Tian and
Yu~\cite{TianYu2023}. They established existence and uniqueness by a continuation method together
with a priori and stability estimates. Related stochastic maximum
principles under generalized monotonicity conditions were subsequently
developed in~\cite{Niu2026b}. In the infinite-horizon discrete-time
setting, Ma, Li, and Meng~\cite{MaLiMeng2025} developed a
domination--monotonicity framework for fully coupled nonlinear
FBS$\Delta$Es whose backward coefficients depend on conditional
expectations of future quantities, and proved unique solvability by a
continuation method. More recently, Chen and Zhang~\cite{ChenZhang2026}
established the well-posedness of fully coupled nonlinear
FBS$\Delta$Es on an infinite horizon by an explicit spectral-contraction
argument in exponentially weighted spaces, replacing domination--monotonicity
and continuation by a computable discount threshold. Together with
\cite{JiLiu2024,Song2023}, these works have substantially advanced the
solvability and control theory of nonlinear, fully coupled
FBS$\Delta$Es without mean-field interactions.

In parallel, a substantial literature has developed for discrete-time
mean-field stochastic control. Elliott, Li, and Ni~\cite{ElliottLiNi2013}
established solvability results for finite-horizon discrete-time
mean-field stochastic linear--quadratic (MF-LQ) control and characterized
optimal feedback laws through coupled Riccati difference equations.
The infinite-horizon counterpart and its stability properties were
subsequently studied by Ni, Elliott, and Li~\cite{NiElliottLi2015}.
Indefinite discrete-time MF-LQ problems were further investigated by
Ni, Zhang, and Li~\cite{NiZhangLi2015}. From a forward--backward
perspective, Song and Liu~\cite{SongLiu2021} characterized open-loop
solvability of a finite-horizon discrete-time MF-LQ problem through the
solvability of the MF-FBS$\Delta$E arising from the stochastic maximum
principle. These results demonstrate that mean-field forward--backward
systems are closely connected with discrete-time optimal control, although
their analysis is primarily tied to linear--quadratic structures.

A parallel strand addresses discrete-time stochastic systems of
forward--backward or recursive type with expectation coupling. Wu and
Zhang~\cite{WuZhang2022} derived a Pontryagin-type maximum principle for
discrete-time stochastic optimal control problems and stochastic games, and
Zhang and Zhang~\cite{ZhangZhang2025} later studied discrete-time
forward--backward stochastic difference equations with control constraints
and their application to a recursive-utility portfolio problem. In the
mean-field setting, Wu, Yu, and Zhang~\cite{WuYuZhang2026} analyzed a
discrete-time linear--quadratic mean-field game with state-, control-, and
mean-dependent diffusion and constructed decentralized $\varepsilon$-Nash
strategies through a discrete-time consistency condition. These works
characterize optimal controls and equilibria for controlled discrete-time
mean-field systems; they do not provide a coefficient-level well-posedness
theory for the underlying nonlinear mean-field FBS$\Delta$E.

Mean-field FBS$\Delta$Es have also been investigated more directly.
Song and Liu~\cite{SongLiu2020MF} established a stochastic maximum
principle for controlled fully coupled forward--backward stochastic
difference systems of mean-field type under weak convexity assumptions.
Their objective was the derivation of necessary and sufficient optimality
conditions rather than a general nonlinear well-posedness theory for the
underlying uncontrolled MF-FBS$\Delta$E. More recently, Ma, Mou and
Ho~\cite{Ma2025} systematically studied fully coupled \emph{linear}
MF-FBS$\Delta$Es. They first showed that three formulations of
discrete-time forward--backward systems available in the literature can be
cast into a unified paradigm, and then used two coupled difference Riccati
equations to obtain a \emph{necessary and sufficient} finite-horizon
well-posedness condition together with explicit solution representations.
On the infinite horizon, stabilizability and detectability yield a
sufficient condition through the stabilizing solution of two coupled
algebraic Riccati equations, and open-loop solvability of finite- and
infinite-horizon MF-LQ problems follows as an application. Thus, the
existing discrete-time mean-field literature provides a mature theory for
MF-LQ control and, more recently, for fully coupled linear
MF-FBS$\Delta$Es.

The result of~\cite{Ma2025} is the closest discrete-time antecedent of the
present one, and the two differ in four respects. \emph{(i)} Their Riccati
mechanism is tied to linear--quadratic structure: the criterion is read off
from the solvability of two matrix difference equations, which degenerates
once the coefficients are nonlinear. Here the conditions are imposed
directly on the coefficients, through domination of frozen differences and
coercivity of centered and mean matrix combinations, and are closed by
continuation rather than by Riccati solvability. \emph{(ii)} For the linear
class, \cite{Ma2025} characterizes well-posedness and gives explicit
solution representations; the assumptions below are sufficient, and no
characterization of all solvable nonlinear coefficient families is claimed
(Section~\ref{sec:scope}). \emph{(iii)} The mean-field coupling is of a
different type. The linear relations used in~\cite{Ma2025} are compatible
with Riccati decoupling, whereas here the coefficients act on the next-step
conditional projections $y'_{k+1}=\mathbb E[y_{k+1}\mid\mathcal F_{k-1}]$
and $z'_{k+1}=\mathbb E[y_{k+1}w_k\mid\mathcal F_{k-1}]$ together with their
first moments, and these same projections must be paired with the frozen
coefficient differences in order to close the estimate. \emph{(iv)} The
results here are finite-horizon and require no full-rank condition on the
coercivity matrices: directions lying in their kernels are controlled by a
companion one-sided estimate. The infinite-horizon treatment
of~\cite{Ma2025}, by contrast, rests on a stabilizing algebraic Riccati
solution.

In continuous time, the mean-field domination--monotonicity framework of
Tian and Yu~\cite{TianYu2023} was extended to McKean--Vlasov
FBSDEs by Chen and Nie~\cite{ChenNie2025} and, in a nonlinear form, to
extended MF-FBSDEs by Wu~\cite{Wu2026}. These advances concern continuous-time
laws and do not transfer directly to the next-step projection structure of
discrete-time systems studied here.

The present paper studies a first-moment mean-field extension of the
nonlinear discrete-time framework of~\cite{Niu2026a}. Its coefficients depend
on $\theta_k=(x_k,y'_{k+1},z'_{k+1})$ and $\mathbb E[\theta_k]$, and its
terminal mapping depends on $x_N$ and $\mathbb E[x_N]$.
Table~\ref{tab:scope} summarizes the scope of the closest results.

The following four points are what the paper establishes beyond those
frameworks.
\begin{itemize}
\item[(N1)] A \emph{deterministic-matrix} domination--monotonicity formulation in
centered and mean coordinates whose coercivity acts on the combinations
$B_k\widehat Y^{(1)}+C_k\widehat Z^{(1)}+\widetilde B_k\widehat Y^{(2)}
+\widetilde C_k\widehat Z^{(2)}$ that are compatible with the next-step
conditional projections. The continuous-time mean-field construction
of~\cite{TianYu2023} pairs coefficients with It\^o martingale integrands; on a
discrete filtration the corresponding objects are
$y'_{k+1}=\mathbb E[y_{k+1}\mid\mathcal F_{k-1}]$ and
$z'_{k+1}=\mathbb E[y_{k+1}w_k\mid\mathcal F_{k-1}]$, and pairing the frozen
coefficient differences with these projections is what closes the estimate.
\item[(N2)] An estimate whose constant is \emph{uniform in the homotopy
parameter}. Normalizing the single active coercivity parameter
(Lemma~\ref{lem:normalization}) suffices, so no additional smallness restriction
is imposed on coefficients that already satisfy the domination and monotonicity
inequalities. Because one coercion direction is retained, a system satisfying
stronger two-sided conditions needs only one of them to be verified.
\item[(N3)] Domination matrices that may be \emph{degenerate}. Directions in
their kernels are not controlled by coercivity but by the companion one-sided
estimate; no full-rank condition on the domination matrices is required.
\item[(N4)] A genuine extension beyond the linear class.
Corollary~\ref{cor:degenerate} shows that switching off the mean-field terms
returns exactly the non-mean-field theory of~\cite{Niu2026a}, whereas
Example~\ref{ex:nonlinear} is a nonlinear mean-field system with degenerate
directions that satisfies the assumptions and lies outside the linear
MF-FBS$\Delta$E theory of~\cite{Ma2025}.
\end{itemize}
Restrictions that remain in force are collected in
Section~\ref{sec:scope}.

\begin{table*}[t]
\centering\small
\caption{Scope of the closest solvability frameworks.}\label{tab:scope}
\begin{tabular}{p{0.12\textwidth}p{0.19\textwidth}p{0.19\textwidth}p{0.19\textwidth}p{0.19\textwidth}}
\hline
Feature & Niu et al.~\cite{Niu2026a}, Ma, Li, Meng~\cite{MaLiMeng2025}, Chen--Zhang~\cite{ChenZhang2026} & Tian--Yu~\cite{TianYu2023}, Chen--Nie~\cite{ChenNie2025} & Ma et al.~\cite{Ma2025} & Present work\\
\hline
Time and system & Discrete; nonlinear, fully coupled & Continuous; nonlinear mean-field & Discrete; linear MF & Discrete; nonlinear MF\\
Interaction & Next-step projections; no mean-field terms & State and mean-field backward variables & Linear MF relations & First moments of projected triples\\
Method & Domination--monotonicity, continuation, or spectral contraction & Mean-field domination--monotonicity & Coupled Riccati equations & Deterministic matrix domination and continuation\\
Solvability claim & Unique solvability and estimates & Unique solvability, estimates, coefficient dependence & Linear necessary/sufficient criterion & Nonlinear sufficient conditions and stability\\
Role in this paper & Discrete projection and duality mechanism & Mean-field domination construction & Closest linear antecedent & Nonlinear extension and LQ illustration\\
\hline
\end{tabular}
\end{table*}

The technical point is the compatibility of two estimates. The discrete duality
identity pairs the diffusion coefficient with
$z'_{k+1}=\mathbb E[y_{k+1}w_k\mid\mathcal F_{k-1}]$, whereas coercivity is
expressed through deterministic matrix combinations of centered and mean
variables. Orthogonality alone bounds these combinations, not the complete
solution. We therefore prove that the same combinations dominate the frozen
forward or backward coefficient differences. The companion one-sided estimate
then controls the remaining solution variables. Conditional projection bounds
close the source pairing, producing an estimate independent of the current
homotopy parameter. No full-rank condition on the domination matrices is needed.
These steps are assembled into the continuation argument of
Theorem~\ref{thm:main}.

The results are as follows.
\begin{itemize}
\item[(i)] We give a deterministic-matrix domination--monotonicity formulation
for nonlinear first-moment MF-FBS$\Delta$Es, including mixed initial and terminal
conditions. Coercivity acts on specified matrix combinations in centered and
mean coordinates; it need not control each complete mode directly.
\item[(ii)] We prove existence, uniqueness, a priori bounds, and coefficient
stability under global Lipschitz continuity and one active coercivity direction.
The continuation step is uniform in the homotopy parameter; its constants depend
on the active parameter through $\mu^{-2}$ (or $\nu^{-2}$), as made explicit
in~\eqref{eq:mu-dependence}--\eqref{eq:nu-dependence}. Systems satisfying
stronger two-sided conditions can be
treated by retaining one direction. The opposite sign orientation follows by
conjugation.
\item[(iii)] We verify the assumptions for a genuinely nonlinear system with
saturating state and mean interactions. A scalar LQ buffer-adjustment model
shows how the theorem yields a Hamiltonian characterization of the unique
open-loop optimizer. An exact binary-tree computation independently compares
the Hamiltonian equations and direct quadratic optimization.
\end{itemize}
Every restriction that remains in force---sufficient rather than necessary
conditions, unconditional rather than conditional means, one active coercivity
direction, and the illustrative status of the LQ computation---is collected in
Section~\ref{sec:scope}.

Section~\ref{sec:formulation} states the system and assumptions.
Section~\ref{sec:results} gives the main results and nonlinear example.
Section~\ref{sec:proofs} proves the continuation estimates.
Section~\ref{sec:application} treats the LQ illustration and computation.
Section~\ref{sec:scope} states the scope and limitations, and
Section~\ref{sec:conclusion} concludes.

\section{System Formulation and Hypotheses}\label{sec:formulation}
Choose a positive finite horizon $N\in\{1,2,\ldots\}$ and abbreviate the running and extended
index sets by $\mathbb T=\{0,\ldots,N-1\}$ and
$\overline{\mathbb T}=\{0,\ldots,N\}$, respectively. All random objects are
defined on a complete space $(\Omega,\mathcal F,\mathbb P)$ equipped with
$\mathbb F=\{\mathcal F_k\}_{k=-1}^{N-1}$. The scalar innovation sequence
$\{w_k\}_{k\in\mathbb T}$ is normalized conditionally:
\begin{flalign*}
	&
	\mathbb{E}\left[ w_k \mid \mathcal{F}_{k-1} \right] = 0,
	\quad
	\mathbb{E}\left[ w_k^2 \mid \mathcal{F}_{k-1} \right] = 1,
	\qquad k \in \mathbb{T}.
	&
\end{flalign*}
Take $\mathcal F_k$ to be the completion of $\sigma(w_0,\ldots,w_k)$,
and $\mathcal F_{-1}$ to be the completed trivial $\sigma$-algebra.
Variables measurable at time $-1$ are deterministic up to almost-sure equality.
Independence of the innovations is not required by the theory.

We use $|\cdot|$ and $\langle\cdot,\cdot\rangle$ for Euclidean geometry and
$\mathbb S^n$ for the symmetric $n\times n$ matrices. On
$\mathbb R^{n\times m}$ the matrix norm and pairing are fixed by
\begin{flalign*}
&|A|=\sqrt{\mathrm{tr}(AA^\top)},
&\langle A,B\rangle=\mathrm{tr}(AB^\top),\quad A,B\in\mathbb R^{n\times m}.&
\end{flalign*}

Given a Hilbert space $(\mathbb H,\|\cdot\|_{\mathbb H})$, the notation below
distinguishes terminal random variables from adapted sequences:

\begin{itemize}
	\item
	$L^2_{\mathcal{F}_{N-1}}(\Omega; \mathbb{H})$: the set of all $\mathbb{H}$-valued $\mathcal{F}_{N-1}$-measurable random variables $\xi$ satisfying
	\begin{flalign*}
		&\hspace*{-\leftmargin}
		\|\xi\|_{L^2_{\mathcal{F}_{N-1}}(\Omega; \mathbb{H})}
		:= \left[ \mathbb{E} \|\xi\|_{\mathbb{H}}^2 \right]^{\frac{1}{2}} < \infty.
		&
	\end{flalign*}
	
	\item
	$L^2_{\mathbb{F}}(\mathbb{T}; \mathbb{H})$: the set of all $\mathbb{H}$-valued stochastic processes $f(\cdot) = \{f_k\mid f_k \text{ is } \mathcal{F}_{k-1}\text{-measurable}, k\in\mathbb{T}\}$ satisfying
	\begin{flalign*}
		&\hspace*{-\leftmargin}
		\|f(\cdot)\|_{L^2_{\mathbb{F}}(\mathbb{T}; \mathbb{H})}
		:= \left[ \mathbb{E}\left( \sum_{k=0}^{N-1} \|f_k\|_{\mathbb{H}}^2 \right) \right]^{\frac{1}{2}} < \infty.
		&
	\end{flalign*}

	\item
	$L^2_{\mathbb{F}}(\overline{\mathbb{T}}; \mathbb{H})$: the set of all processes $v(\cdot)=\{v_k\}_{k=0}^{N}$ such that $v_k$ is $\mathcal{F}_{k-1}$-measurable for every $k\in\overline{\mathbb{T}}$ and
	\begin{flalign*}
		&\hspace*{-\leftmargin}
		\|v(\cdot)\|_{L^2_{\mathbb{F}}(\overline{\mathbb{T}}; \mathbb{H})}
		:=\left[\mathbb E\sum_{k=0}^{N}\|v_k\|_{\mathbb H}^2\right]^{1/2}<\infty.
		&
	\end{flalign*}
	Here $v_0$ is deterministic because $\mathcal F_{-1}=\{\emptyset,\Omega\}$, whereas $v_N$ is $\mathcal F_{N-1}$-measurable.
	
    \item
    $\mathbb{U}$: the set of all admissible controls
{\small
    \begin{flalign*}
    	&\hspace*{-\leftmargin}
    \mathbb{U} :=\Bigg\{\mathbf{u} = (u_0, u_1, \dots, u_{N-1})
    \,\Bigg|\,
    u_k \in \mathbb{R}^m, \\
    &\qquad \quad u_k \text{ is } \mathcal{F}_{k-1}\text{-measurable},\quad
    \mathbb{E}\left[ \sum_{k=0}^{N-1} |u_k|^2 \right] < \infty
    \Bigg\}.
    &
    \end{flalign*}
}
\end{itemize}

The continuation argument uses three product spaces. Their superscript $2$
indicates square integrability, not the number of components:
\begin{itemize}
	\item
	$N^2(\overline{\mathbb{T}}; \mathbb{R}^{2n})
	:= L^2_{\mathbb{F}}(\overline{\mathbb{T}}; \mathbb{R}^n)
	\times L^2_{\mathbb{F}}(\overline{\mathbb{T}}; \mathbb{R}^n)$.
	The product norm assigned to
	$(x(\cdot),y(\cdot))\in N^2(\overline{\mathbb T};\mathbb R^{2n})$ is
	\begin{flalign*}
		&\hspace*{-\leftmargin}
		\begin{aligned}
			&\|(x(\cdot), y(\cdot))\|_{N^2(\overline{\mathbb{T}}; \mathbb{R}^{2n})}\\
			&:= \left\{
			\mathbb{E}\left[
			\sum_{k=0}^{N} |x(k)|^2
			+ \sum_{k=0}^{N} |y(k)|^2
			\right]
			\right\}^{\frac{1}{2}}.
		\end{aligned}
		&
	\end{flalign*}
	
	\item
	$\mathcal{N}^2(\mathbb{T}; \mathbb{R}^{3n})
	:= L^2_{\mathbb{F}}(\mathbb{T}; \mathbb{R}^n)
	\times L^2_{\mathbb{F}}(\mathbb{T}; \mathbb{R}^n)
	\times L^2_{\mathbb{F}}(\mathbb{T}; \mathbb{R}^n)$.
	Writing $\rho=(\varphi,\psi,\gamma)$, equip
	$\mathcal N^2(\mathbb T;\mathbb R^{3n})$ with
	\begin{flalign*}
		&\hspace*{-\leftmargin}
		\begin{aligned}
			&\|\rho(\cdot)\|_{\mathcal{N}^2(\mathbb{T}; \mathbb{R}^{3n})}\\
			&:= \Bigg\{
			\mathbb{E}\Bigg[
			\sum_{k=0}^{N-1} |\varphi(k)|^2
			+ \sum_{k=0}^{N-1} |\psi(k)|^2 \\
			&\quad + \sum_{k=0}^{N-1} |\gamma(k)|^2
			\Bigg]
			\Bigg\}^{\frac{1}{2}}.
		\end{aligned}
		&
	\end{flalign*}
	
	\item
	$\mathcal{H}(\overline{\mathbb{T}})
	:= \mathbb{R}^n \times L^2_{\mathcal{F}_{N-1}}(\Omega; \mathbb{R}^n)
	\times \mathcal{N}^2(\mathbb{T}; \mathbb{R}^{3n})$.
	Here \(\xi\) and \(\eta\) represent, respectively, the initial and terminal forcing used during continuation.
	The forcing triple is measured through
{\small
	\begin{flalign*}
		&\hspace*{-\leftmargin}
		\begin{aligned}
			&\|(\xi, \eta, \rho(\cdot))\|_{\mathcal{H}(\overline{\mathbb{T}})}\\
			&:= \Bigg\{
			\|\xi\|^2 + \|\eta\|_{L^2_{\mathcal{F}_{N-1}}(\Omega; \mathbb{R}^n)}^2
			+ \|\rho(\cdot)\|_{\mathcal{N}^2(\mathbb{T}; \mathbb{R}^{3n})}^2
			\Bigg\}^{\frac{1}{2}}.
		\end{aligned}
		&
	\end{flalign*}
}
\end{itemize}

\subsection{Coupled mean-field recursion}
The theorem will be stated before the estimates used to prove it. We therefore
fix the complete coupled model and its structural hypotheses at the outset.

Mutual dependence is now restored: $b$ and $\sigma$ depend on $(y'_{k+1},z'_{k+1})$, whereas $f$ also depends on $x_k$. With $\theta_k=(x_k,y'_{k+1},z'_{k+1})$, the system reads
\begin{flalign}\label{11}
	&
	\begin{cases}
		x_{k+1} = b(k, \theta_k, \mathbb{E}[\theta_k]) + \sigma(k, \theta_k, \mathbb{E}[\theta_k]) w_k, \\
		y_k = -f(k+1, \theta_k, \mathbb{E}[\theta_k]), \\
		x_0 = \Lambda(y_0), \\
		y_N = \Phi(x_N, \mathbb{E}[x_N]),\quad k \in \mathbb{T}.
	\end{cases}
	&
\end{flalign}

Here $\theta_k=(x_k,y'_{k+1},z'_{k+1})$. The projections are defined below;
$z_{k+1}=y_{k+1}w_k$ is auxiliary and is not an additional unknown. For
$y_{k+1}\in L^2$ this product is in $L^1$ by Cauchy--Schwarz, but need not be
in $L^2$. Conditional Cauchy--Schwarz and the conditional variance normalization
give $z'_{k+1}\in L^2$. The backward sign follows the Hamiltonian convention.
All state variables and coefficient values lie in $\mathbb R^n$.

Throughout this section, for any solution pair $(x,y)$, we use the notation
\begin{flalign*}
&\begin{aligned}
\theta_k&=(x_k,y'_{k+1},z'_{k+1}),\\
y'_{k+1}&=\mathbb E[y_{k+1}\mid\mathcal F_{k-1}],\\
z'_{k+1}&=\mathbb E[y_{k+1}w_k\mid\mathcal F_{k-1}].
\end{aligned}&
\end{flalign*}
For two solution pairs, a hat denotes their difference. The coefficient hypotheses follow.

\begin{assumption}\label{ass:lipschitz}
	The maps $\Lambda:\mathbb R^n\to\mathbb R^n$, $\Phi:\Omega\times\mathbb R^n\times\mathbb R^n\to\mathbb R^n$, and the running coefficients are jointly measurable in their random and Euclidean arguments. There is a single measurable set of probability one on which all pointwise
Lipschitz inequalities hold simultaneously for every time and all Euclidean
arguments. All evaluated random variables are identified up to almost-sure equality.

	(i) For any $y\in\mathbb{R}^n$, $\Lambda(y)$ is deterministic. For any $x_1,x_2\in\mathbb R^n$, $\Phi(x_1,x_2)\in L^2_{\mathcal{F}_{N-1}}(\Omega;\mathbb{R}^n)$; for deterministic Euclidean arguments, each running coefficient at time $k$ is $\mathcal F_{k-1}$-measurable.
	Furthermore, for any $\theta_1, \theta_2 \in \mathbb{R}^{3n}$, set $\Gamma(k,\theta_1,\theta_2)=\bigl(f(k+1,\theta_1,\theta_2),\ b(k,\theta_1,\theta_2),\ \sigma(k,\theta_1,\theta_2)\bigr)$, then $\Gamma(\cdot,\theta_1,\theta_2)\in \mathcal{N}^2(\mathbb{T};\mathbb{R}^{3n})$. Moreover, $(\Lambda(0),\Phi(0,0),\linebreak
	\Gamma(\cdot,0,0))\in \mathcal{H}(\overline{\mathbb{T}})$.
	
	(ii) A single \(L>0\), independent of \(k\), bounds the increments of
	\(\Lambda,\Phi,\Gamma\): for all admissible arguments,
	\begin{flalign*}
		&
		\begin{cases}
			|\Lambda(y)-\Lambda(\bar{y})|\le L|\hat{y}|,\\
			|\Phi(x_1,x_2)-\Phi(\bar{x}_1,\bar{x}_2)|\le L\bigl(|\hat{x}_1|+|\hat{x}_2|\bigr),\\
			|h(k,\theta_1,\theta_2)-h(k,\bar{\theta}_1,\bar{\theta}_2)|\le L\bigl(|\hat{\theta}_1|+|\hat{\theta}_2|\bigr),\\
			|f(k+1,\theta_1,\theta_2)-f(k+1,\bar{\theta}_1,\bar{\theta}_2)|\le L\bigl(|\hat{\theta}_1|+|\hat{\theta}_2|\bigr),
		\end{cases}
		&
	\end{flalign*}
	where $h=b,\sigma$. Set $\hat{y}=y-\bar{y}$, with the same notation for $\hat{x}_1$, $\hat{x}_2$, $\hat{\theta}_1$, $\hat{\theta}_2$.

\end{assumption}

The domination structure is stated in centered and mean coordinates,
$X^{(1)}=X-\mathbb E[X]$ and $X^{(2)}=\mathbb E[X]$, so that the specified fluctuation and population matrix combinations can be controlled. Its matrices are deterministic, and the controlled combinations may be degenerate. In particular, for deterministic $P,Q$ the centered and mean components are orthogonal,
\begin{flalign*}
	&\quad
	\mathbb{E}\left\langle P(X-\mathbb{E}[X]),Q\mathbb{E}[X]\right\rangle=0,
	&
\end{flalign*}
an identity used repeatedly below. Section~\ref{sec:application} supplies a Hamiltonian realization.

\begin{assumption}\label{ass:dm}
	There exist two constants $\mu\geq0$, $\nu\geq0$, deterministic matrices $M\in\mathbb{R}^{m_1\times n}$, $G,\widetilde{G}\in\mathbb{R}^{m_2\times n}$, and deterministic bounded matrix sequences
	\begin{flalign*}
		&\quad
	    A_k,\widetilde{A}_k,B_k,\widetilde{B}_k,C_k,\widetilde{C}_k\in\mathbb{R}^{m_3\times n},\qquad k\in\mathbb{T},
		&
	\end{flalign*}
	where $m_1,m_2,m_3\in\mathbb{N}$ are given, such that $\mu+\nu>0$. For any square-integrable random variable $X$, define $X^{(1)}:=X-\mathbb{E}[X]$ and $X^{(2)}:=\mathbb{E}[X]$. For $X,Y,Z$ and $\bar{X},\bar{Y},\bar{Z}$, set $\widehat{X}=X-\bar{X}$, $\widehat{Y}=Y-\bar{Y}$, $\widehat{Z}=Z-\bar{Z}$, and $\widehat{X}^{(1)}:=\widehat{X}-\mathbb{E}[\widehat{X}]$, $\widehat{X}^{(2)}:=\mathbb{E}[\widehat{X}]$, with the same notation for $\widehat{Y}$ and $\widehat{Z}$.
	
	(i) Either ``Case A: $\mu>0$, $\nu=0$'', or ``Case B: $\mu=0$, $\nu>0$''. This unilateral condition ensures that the equation at $\alpha=0$ in the continuation method is decoupled and hence directly solvable.
	
	(ii) (Domination conditions)\\
	If $\mu>0$ and $\nu=0$, then for all $y,\bar{y}\in\mathbb{R}^n$,
	\begin{flalign*}
		&\quad
		|\Lambda(y)-\Lambda(\bar{y})|
		\leq \frac{1}{\mu}|M(y-\bar{y})|,
		&
	\end{flalign*}
	and, for $h=b,\sigma$, for all $\Theta=(X,Y,Z)$ and $\bar\Theta=(X,\bar Y,\bar Z)$ in $L^2(\Omega;\mathbb{R}^{3n})$, with the same first component $X$:
	\begin{flalign*}
		&\quad
		\begin{aligned}
			&\mathbb{E}\left[|h(k,\Theta,\mathbb{E}[\Theta])-h(k,\bar\Theta,\mathbb{E}[\bar\Theta])|^2\right]\\
			&\leq \frac{1}{\mu^2}\mathbb{E}\left[|B_k\widehat{Y}^{(1)}+C_k\widehat{Z}^{(1)}+
			\widetilde{B}_k\widehat{Y}^{(2)}+\widetilde{C}_k\widehat{Z}^{(2)}|^2\right].
		\end{aligned}
		&
	\end{flalign*}
	If $\nu>0$ and $\mu=0$, then for all $X,\bar X\in L^2(\Omega;\mathbb{R}^n)$,
	\begin{flalign*}
		&\quad
		\begin{aligned}
			&\mathbb{E}\left[|\Phi(X,\mathbb{E}[X])-\Phi(\bar X,\mathbb{E}[\bar X])|^2\right]\\
			&\leq \frac{1}{\nu^2}\mathbb{E}\left[|G\widehat{X}^{(1)}+\widetilde{G}\widehat{X}^{(2)}|^2\right],
		\end{aligned}
		&
	\end{flalign*}
	and, for $\Theta=(X,Y,Z)$ and $\bar\Theta=(\bar X,Y,Z)$,
	\begin{flalign*}
		&\quad
		\begin{aligned}
		&\mathbb{E}\left[|f(k+1,\Theta,\mathbb{E}[\Theta])-f(k+1,\bar\Theta,\mathbb{E}[\bar\Theta])|^2\right]\\
		&\leq \frac{1}{\nu^2}\mathbb{E}\left[|A_k\widehat{X}^{(1)}+\widetilde{A}_k\widehat{X}^{(2)}|^2\right].
		\end{aligned}
		&
	\end{flalign*}
	
	(iii) (Monotonicity conditions) For all $\Theta=(X,Y,Z)$, $\bar\Theta=(\bar X,\bar Y,\bar Z)\in L^2(\Omega;\mathbb{R}^{3n})$ and $y,\bar y\in\mathbb{R}^n$, a single one of the following two alternatives is assumed globally, with the same choice for every time and every pair of arguments.\\
	\emph{First monotonicity alternative:}
	\begin{flalign*}
		\begin{cases}
			&\langle\Lambda(y)-\Lambda(\bar y),y-\bar y\rangle\leq-\mu|M(y-\bar y)|^2,\\[0.5em]
			&\mathbb{E}\left[\left\langle\Phi(X,\mathbb{E}[X])-\Phi(\bar X,\mathbb{E}[\bar X]),\widehat X\right\rangle\right]\\
			&\quad\geq \nu\mathbb{E}\left[|G\widehat X^{(1)}+\widetilde G\widehat X^{(2)}|^2\right],\\[0.5em]
			&\mathbb{E}\left[\left\langle\Gamma(k,\Theta,\mathbb{E}[\Theta])-\Gamma(k,\bar\Theta,\mathbb{E}[\bar\Theta]),\widehat\Theta\right\rangle\right]\\
			&\quad\leq-\nu\mathbb{E}\left[|A_k\widehat X^{(1)}+\widetilde A_k\widehat X^{(2)}|^2\right]\\
			&\quad-\mu\mathbb{E}\left[|B_k\widehat Y^{(1)}+C_k\widehat Z^{(1)}+
			\widetilde B_k\widehat Y^{(2)}+\widetilde C_k\widehat Z^{(2)}|^2\right].
		\end{cases}
	\end{flalign*}\\
	\emph{Second monotonicity alternative:}
	\begin{flalign*}
		\begin{cases}
		&\langle\Lambda(y)-\Lambda(\bar y),y-\bar y\rangle\geq\mu|M(y-\bar y)|^2,\\[0.5em]
		&\mathbb{E}\left[\left\langle\Phi(X,\mathbb{E}[X])-\Phi(\bar X,\mathbb{E}[\bar X]),\widehat X\right\rangle\right]\\
		&\quad\leq-\nu\mathbb{E}\left[|G\widehat X^{(1)}+\widetilde G\widehat X^{(2)}|^2\right],\\[0.5em]
		&\mathbb{E}\left[\left\langle\Gamma(k,\Theta,\mathbb{E}[\Theta])-\Gamma(k,\bar\Theta,\mathbb{E}[\bar\Theta]),\widehat\Theta\right\rangle\right]\\
		&\quad\geq\nu\mathbb{E}\left[|A_k\widehat X^{(1)}+\widetilde A_k\widehat X^{(2)}|^2\right]\\
		&\quad+\mu\mathbb{E}\left[|B_k\widehat Y^{(1)}+C_k\widehat Z^{(1)}+
		\widetilde B_k\widehat Y^{(2)}+\widetilde C_k\widehat Z^{(2)}|^2\right].
		\end{cases}
	\end{flalign*}
	Here
	\begin{flalign*}
		&\quad
		\begin{aligned}
		&\left\langle\Gamma(k,\Theta,\mathbb{E}[\Theta])-\Gamma(k,\bar\Theta,\mathbb{E}[\bar\Theta]),\widehat\Theta\right\rangle\\
		&:=\left\langle f(k+1,\Theta,\mathbb{E}[\Theta])-f(k+1,\bar\Theta,\mathbb{E}[\bar\Theta]),\widehat X\right\rangle\\
		&\quad+\left\langle b(k,\Theta,\mathbb{E}[\Theta])-b(k,\bar\Theta,\mathbb{E}[\bar\Theta]),\widehat Y\right\rangle\\
		&\quad+\left\langle \sigma(k,\Theta,\mathbb{E}[\Theta])-\sigma(k,\bar\Theta,\mathbb{E}[\bar\Theta]),\widehat Z\right\rangle.
		\end{aligned}
		&
	\end{flalign*}
	
	(iv) (Structural size condition) All matrix norms below are understood as operator norms. For the
	deterministic matrix sequences in Assumption~\ref{ass:dm}, set
	\begin{flalign*}
		&\quad
		\begin{aligned}
			\kappa_M &:= \|M\|,\\
			\kappa_G &:= \sqrt{2\bigl(\|G\|^2+\|\widetilde G\|^2\bigr)},\\
			\kappa_A &:= \sqrt{2\bigl(\|A\|_\infty^2+
				\|\widetilde A\|_\infty^2\bigr)},\\
			\kappa_B &:= \sqrt{2\bigl(\|B\|_\infty^2+
				\|\widetilde B\|_\infty^2\bigr)},\\
			\kappa_C &:= \sqrt{2\bigl(\|C\|_\infty^2+
				\|\widetilde C\|_\infty^2\bigr)},
		\end{aligned}
		&
	\end{flalign*}
	where
	\begin{flalign*}
		&\quad
		\|A\|_\infty:=\sup_{k\in\mathbb T}\|A_k\|,
		\quad
		\|\widetilde A\|_\infty:=\sup_{k\in\mathbb T}\|\widetilde A_k\|,
		&
	\end{flalign*}
	and the other norms \(\|B\|_\infty\), \(\|\widetilde B\|_\infty\),
	\(\|C\|_\infty\), and \(\|\widetilde C\|_\infty\) are defined similarly.\\
	In Case A, namely \(\mu>0\) and \(\nu=0\), we assume that
	\begin{flalign*}
		&\quad
		\mu^2 \max\{\kappa_M,\kappa_B,\kappa_C\}\le 1.
		&
	\end{flalign*}
	In Case B, namely \(\mu=0\) and \(\nu>0\), we assume that
	\begin{flalign*}
		&\quad
		\nu^2 \max\{\kappa_G,\kappa_A\}\le 1.
		&
	\end{flalign*}
	These bounds normalize the active parameter so that the same domination
constants are valid along continuation. They impose no additional smallness
restriction on coefficients already satisfying (i)--(iii); see
Lemma~\ref{lem:normalization}. Absorption of solution energies is a separate
step in Lemma~\ref{lem:apriori}.
\end{assumption}

\begin{remark}[Choice of one coercivity direction]\label{rem:unilateral}
The proof retains one active parameter to make the continuation anchor
sequentially solvable: backward then forward in Case~A, and forward then
backward in Case~B. This does not exclude systems satisfying stronger conditions
with both $\mu_0,\nu_0>0$. Suppose such a system satisfies the corresponding
domination inequalities on both sides and the displayed monotonicity inequality
with both coercive terms. Setting $\nu=0$ and retaining $0<\mu\le\mu_0$
preserves Case~A domination, terminal monotonicity, and running monotonicity;
the discarded term has the favorable sign. Alternatively, set $\mu=0$ and
retain $0<\nu\le\nu_0$. The same observation applies to the second orientation.
Keeping both parameters in the particular anchor~\eqref{17} would leave a
coupled base system, but this is unnecessary for applying the present sufficient
conditions.

Discarding one direction also preserves the remaining hypotheses, so the
reduction is not merely a formal device. Dropping the $\nu$-terms only deletes
$\nu$-domination and the two $\nu$-monotonicity inequalities, which are not
required in Case~A; the $\mu$-inequalities are inherited unchanged. The
structural size condition~(iv) is then still available: after shrinking the
active parameter to any
$c\le\min\{\mu_0,(1+H)^{-1/2}\}$ with $H=\max\{\kappa_M,\kappa_B,\kappa_C\}$
the domination factor $1/c$ widens, the required coercivity $c$ decreases, and
the terminal bound is preserved, so Lemma~\ref{lem:normalization} applies and
Assumption~\ref{ass:dm}(i)--(iv) holds with the same coefficients. The
symmetric statement holds for Case~B with $H=\max\{\kappa_G,\kappa_A\}$. Only
the constants in the solution estimates change.
\end{remark}

\begin{lemma}[Normalization of the active parameter]\label{lem:normalization}
Suppose Assumption~\ref{ass:dm}(i)--(iii) holds with fixed bounded matrices
and an active parameter $c_0>0$. In Case~A let
$H=\max\{\kappa_M,\kappa_B,\kappa_C\}$; in Case~B let
$H=\max\{\kappa_G,\kappa_A\}$. For any
\[
0<c\le\min\{c_0,(1+H)^{-1/2}\},
\]
the same coefficients and matrices satisfy (i)--(iv) with active parameter $c$.
\end{lemma}
\begin{proof}
Reducing $c_0$ to $c$ increases the domination factor $1/c$ and decreases the
magnitude of the required coercivity $c$. Thus all inequalities are preserved
in either orientation. Also $c^2H\le H/(1+H)\le1$. This choice changes the
constants in the solution estimates; no uniform bound as $c\downarrow0$ is claimed.
\end{proof}

\begin{remark}[Reduction of the second monotonicity alternative]\label{rem:sign}
	Only the first alternative must be proved directly. The involution below
	converts the second one into it.\\
	Define
	\begin{flalign*}
		&
		\mathcal S(x,y,z)=(x,-y,-z).
		&
	\end{flalign*}
	For coefficients in the second orientation, set
	\begin{flalign*}
		&
		\widetilde\Lambda(y):=\Lambda(-y),\qquad
		\widetilde\Phi(x,\bar x):=-\Phi(x,\bar x),
		&
	\end{flalign*}
	and
	\begin{flalign*}
		&
		\begin{aligned}
			\widetilde\Gamma(k,\theta,\bar\theta)
			&:=
			\begin{pmatrix}
				\widetilde f(k+1,\theta,\bar\theta)\\
				\widetilde b(k,\theta,\bar\theta)\\
				\widetilde\sigma(k,\theta,\bar\theta)
			\end{pmatrix}\\&
			:=
			\begin{pmatrix}
				- f(k+1,\mathcal S\theta,\mathcal S\bar\theta)\\
				b(k,\mathcal S\theta,\mathcal S\bar\theta)\\
				\sigma(k,\mathcal S\theta,\mathcal S\bar\theta)
			\end{pmatrix}.
		\end{aligned}
		&
	\end{flalign*}
	For
	\begin{flalign*}
		&
		(\widetilde x_k,\widetilde y_k):=(x_k,-y_k),
		&
	\end{flalign*}
	linearity of both projections gives
	\begin{flalign*}
		&
		\widetilde y'_{k+1}=-y'_{k+1},\qquad
		\widetilde z'_{k+1}=-z'_{k+1}.
		&
	\end{flalign*}
	The transformed MF-FBS$\Delta$E therefore has coefficients
	\begin{flalign*}
		&
		(\widetilde\Lambda,\widetilde\Phi,\widetilde\Gamma)
		&
	\end{flalign*}
	and is solved by
	\begin{flalign*}
		&
		(\widetilde x(\cdot),\widetilde y(\cdot))
		=(x(\cdot),-y(\cdot))
		&
	\end{flalign*}
	if and only if $(x(\cdot),y(\cdot))$ solves the original system. For the initial map,
	\begin{flalign*}
		&
		\begin{aligned}
			&\left\langle
			\widetilde\Lambda(y)-\widetilde\Lambda(\bar y),
			y-\bar y
			\right\rangle\\
			&=
			-\left\langle
			\Lambda(-y)-\Lambda(-\bar y),
			(-y)-(-\bar y)
			\right\rangle  \\
			&\leq -\mu |M(y-\bar y)|^2 .
		\end{aligned}
		&
	\end{flalign*}
	For the terminal map,
	\begin{flalign*}
		&
		\begin{aligned}
			&\mathbb E
			\left[
			\left\langle
			\widetilde\Phi(X,\mathbb E[X])
			-\widetilde\Phi(\bar X,\mathbb E[\bar X]),
			X-\bar X
			\right\rangle
			\right]  \\
			&
			=
			-\mathbb E
			\left[
			\left\langle
			\Phi(X,\mathbb E[X])
			-\Phi(\bar X,\mathbb E[\bar X]),
			X-\bar X
			\right\rangle
			\right]  \\
			&
			\geq
			\nu\,
			\mathbb E
			\left[
			|G\widehat X^{(1)}+\widetilde G\widehat X^{(2)}|^2
			\right].
		\end{aligned}
		&
	\end{flalign*}
	Moreover, since
	\begin{flalign*}
		&
		\mathcal S\widehat\Theta=(\widehat X,-\widehat Y,-\widehat Z),
		&
	\end{flalign*}
	the $\Gamma$ pairing changes sign as well. In Case~A, sign changes in the
	\(Y,Z\) domination expression disappear after squaring; in Case~B the same is
	true for the transformed \(f\)-difference. The bounds for
	\(\widetilde\Lambda\) and \(\widetilde\Phi\) are identical to the original
	ones, and \(\mathcal S\) preserves Lipschitz and integrability properties.
	Hence the transformed coefficients satisfy the first alternative. After
	applying Theorem~\ref{thm:main}, invert the change by
	\begin{flalign*}
		&
		x_k=\widetilde x_k,\qquad y_k=-\widetilde y_k .
		&
	\end{flalign*}
	No parity assumption on the coefficients is used.
\end{remark}

\section{Main Results}\label{sec:results}
\subsection{Principal solvability statement}
The following sufficient solvability result is proved in Section~\ref{sec:proofs}.

\begin{theorem}\label{thm:main}
Suppose $(\Lambda,\Phi,\Gamma)$ obey Assumptions~\ref{ass:lipschitz}
and~\ref{ass:dm}. Then~\eqref{11} is well posed in
$N^2(\overline{\mathbb T};\mathbb R^{2n})$: there is one and only one pair
$(x(\cdot),y(\cdot))$ in this space, and it satisfies
\begin{flalign}\label{12}
&\mathbb E\left[\sum_{k=0}^{N}|x_k|^2+\sum_{k=0}^{N}|y_k|^2\right]\le K\mathbb E[I],&
\end{flalign}
where
\begin{flalign}\label{13}
&\begin{aligned}
I={}&|\Phi(0,0)|^2+|\Lambda(0)|^2
+\sum_{k=0}^{N-1}|b(k,0,0)|^2\\
&+\sum_{k=0}^{N-1}\bigl(|\sigma(k,0,0)|^2+|f(k+1,0,0)|^2\bigr).
\end{aligned}&
\end{flalign}
Here $K>0$ is determined solely by $N$, the common Lipschitz bound,
$\mu,\nu$, and the deterministic matrix bounds in Assumption~\ref{ass:dm}.

More generally, let $(\bar x,\bar y)\in
N^2(\overline{\mathbb T};\mathbb R^{2n})$ be a solution of~\eqref{11}
for another coefficient set $(\bar\Lambda,\bar\Phi,\bar\Gamma)$.
Assume that the evaluated coefficient differences in~\eqref{15} define a
source triple in $\mathcal H(\overline{\mathbb T})$: the initial difference
is deterministic, the terminal difference is $\mathcal F_{N-1}$-measurable,
and the running differences are $\mathcal F_{k-1}$-measurable and square integrable.
The second set need not have the same domination matrices or monotonicity
orientation. Then $(\bar x,\bar y)$ and $(x,y)$ satisfy
\begin{flalign}\label{14}
&\mathbb E\left[\sum_{k=0}^{N}|\widehat x_k|^2+\sum_{k=0}^{N}|\widehat y_k|^2\right]\le K\mathbb E[\widehat I],&
\end{flalign}
where
\begin{flalign}\label{15}
&\begin{aligned}
\widehat I={}&|\Phi(\bar x_N,\mathbb E[\bar x_N])-\bar\Phi(\bar x_N,\mathbb E[\bar x_N])|^2\\
&+|\Lambda(\bar y_0)-\bar\Lambda(\bar y_0)|^2\\
&+\sum_{k=0}^{N-1}|b(k,\bar\theta_k,\mathbb E[\bar\theta_k])-\bar b(k,\bar\theta_k,\mathbb E[\bar\theta_k])|^2\\
&+\sum_{k=0}^{N-1}|\sigma(k,\bar\theta_k,\mathbb E[\bar\theta_k])-\bar\sigma(k,\bar\theta_k,\mathbb E[\bar\theta_k])|^2\\
&+\sum_{k=0}^{N-1}|f(k+1,\bar\theta_k,\mathbb E[\bar\theta_k])-\bar f(k+1,\bar\theta_k,\mathbb E[\bar\theta_k])|^2.
\end{aligned}&
\end{flalign}
Here $\widehat x=x-\bar x$, $\widehat y=y-\bar y$, and $\bar\theta_k=(\bar x_k,\bar y'_{k+1},\bar z'_{k+1})$.
\end{theorem}

\begin{corollary}[Reduction to the non-mean-field setting]\label{cor:degenerate}
Suppose the coefficients carry no mean-field coupling: $b,\sigma,f$ do not
depend on $\mathbb E[\theta_k]$ and $\Phi$ does not depend on $\mathbb E[x_N]$.
Then $\widetilde A_k=\widetilde B_k=\widetilde C_k=0$ and $\widetilde G=0$
satisfy Assumption~\ref{ass:dm}(ii)--(iii) if and only if the centered
inequalities
\begin{flalign*}
&\quad
\begin{aligned}
|\Lambda(y)-\Lambda(\bar y)|
&\le \tfrac1\mu|M(y-\bar y)|,\\
\mathbb E|h(k,\Theta)-h(k,\bar\Theta)|^2
&\le \tfrac1{\mu^2}\mathbb E|B_k\widehat Y^{(1)}+C_k\widehat Z^{(1)}|^2,\\
&\qquad h=b,\sigma,
\end{aligned}
&
\end{flalign*}
hold whenever $\mu>0$ (with $y'_{k+1},z'_{k+1}$ still given by the conditional
projections), and system~\eqref{11} becomes a fully coupled nonlinear
FBS$\Delta$E without mean-field terms. Under these conditions
Theorem~\ref{thm:main} reduces to the solvability, a priori bound, and
coefficient-stability statement of the non-mean-field domination--monotonicity
framework of~\cite{Niu2026a}: once the mean-field blocks vanish, the displayed
deterministic-matrix conditions are exactly the centered form of the domination
and monotonicity inequalities used there, and the size condition~(iv) is simply
the normalization of the single active parameter supplied by
Lemma~\ref{lem:normalization}. Consequently the mean-field terms
$\widetilde B_k,\widetilde C_k,\widetilde G$ are exactly the additional degrees
of freedom that the present theory adds to \cite{Niu2026a}, and
Example~\ref{ex:nonlinear} exhibits a system in which they are nonzero and the
extra generality is used.
\end{corollary}
\begin{proof}
Setting the mean-field matrices to zero removes the mean-field arguments from
Assumption~\ref{ass:dm}(ii)--(iii), leaving precisely the centered domination
and monotonicity inequalities displayed above together with the terminal
inequality scaled by $\nu$. The forward and backward recursions in~\eqref{11}
then no longer involve $\mathbb E[\theta_k]$, so the projection structure and
the frozen-coefficient comparison of Section~\ref{sec:proofs} are unchanged;
only the mean-field blocks of the domination matrices vanish. Lemma~\ref{lem:preservation},
Lemma~\ref{lem:apriori}, and Lemma~\ref{lem:continuation} therefore apply with
the same constants, and Theorem~\ref{thm:main} yields the stated conclusions.
\end{proof}

\subsection{A nonlinear mean-field example}
\begin{example}[Saturating interactions and degenerate combinations]\label{ex:nonlinear}
Let $n=1$, $N$ be any finite horizon, $x_{\rm init}\in\mathbb R$, and
$d,\delta\in L^2_{\mathbb F}(\mathbb T;\mathbb R)$. For pointwise arguments
$\theta=(x,y,z)$ and $\zeta=(\bar x,\bar y,\bar z)$ set
\[
U_1=(y-\bar y)+(z-\bar z),\qquad U_2=2\bar y+\bar z.
\]
Define
\begin{equation}\label{eq:nonlinear-example}
\begin{aligned}
b(k,\theta,\zeta)&=0.7x+d_k-U_1-2U_2,\\
\sigma(k,\theta,\zeta)&=\delta_k-U_1-U_2,\\
f(k+1,\theta,\zeta)&=-0.7y-\tanh x-\tanh\bar x,\\
\Lambda(y)&=x_{\rm init},\\
\Phi(x,\bar x)&=\tanh x+2\tanh\bar x.
\end{aligned}
\end{equation}
The system is fully coupled: both forward coefficients depend on the projected
backward variables, and the backward recursion and terminal condition depend
nonlinearly on the state and its mean.
\end{example}

\begin{proof}[Verification]
Since $\tanh$ is increasing and globally $1$-Lipschitz, all coefficients satisfy
Assumption~\ref{ass:lipschitz}. Take
\[
\mu=\tfrac12,\quad\nu=0,\quad M=G=\widetilde G=A_k=\widetilde A_k=0,
\]
and $B_k=C_k=\widetilde C_k=1$, $\widetilde B_k=2$ for all $k$.
Then $\kappa_B=\sqrt{10}$, $\kappa_C=2$, and
$\mu^2\max\{\kappa_M,\kappa_B,\kappa_C\}=\sqrt{10}/4<1$.
For a difference with $X$ fixed, put
\[
V_1=\widehat Y^{(1)}+\widehat Z^{(1)},\qquad
V_2=2\widehat Y^{(2)}+\widehat Z^{(2)}.
\]
$V_1$ is centered and $V_2$ is deterministic. Consequently,
\begin{align*}
\mathbb E|\Delta b|^2&=\mathbb E|V_1|^2+4|V_2|^2
 \le4\mathbb E|V_1+V_2|^2,\\
\mathbb E|\Delta\sigma|^2&=\mathbb E|V_1+V_2|^2.
\end{align*}
These are exactly the Case~A domination inequalities. Initial domination is
zero on both sides. For arbitrary differences write
$m=\mathbb E[X]$, $\bar m=\mathbb E[\bar X]$. Cancellation of the $0.7$ terms
and centered/mean orthogonality give
\begin{align*}
\mathbb E\langle\Delta\Gamma,\widehat\Theta\rangle
={}&-\mathbb E[(\tanh X-\tanh\bar X)\widehat X]\\
 &-(\tanh m-\tanh\bar m)(m-\bar m)\\
 &-\mathbb E|V_1|^2-|V_2|^2\\
\le{}&-\mu\mathbb E|V_1+V_2|^2.
\end{align*}
Likewise,
\begin{align*}
\mathbb E[\Delta\Phi\widehat X]
={}&\mathbb E[(\tanh X-\tanh\bar X)\widehat X]\\
&+2(\tanh m-\tanh\bar m)(m-\bar m)\ge0.
\end{align*}
Thus the first orientation holds. Theorem~\ref{thm:main} gives a unique
adapted square-integrable solution for every finite $N$. The nonlinear state
and first-moment terms prevent this system from being a special case of the
linear MF-FBS$\Delta$E theory in~\cite{Ma2025}; its nonzero mean interactions
also lie outside the non-mean-field class in~\cite{Niu2026a}.
\end{proof}

The example also explains the meaning of degeneracy. Centered perturbations
$\widehat Z^{(1)}=-\widehat Y^{(1)}$ and mean perturbations
$\widehat Z^{(2)}=-2\widehat Y^{(2)}$ make both $V_1$ and $V_2$ vanish.
Coercivity therefore does not directly control all $(Y,Z)$ directions.
However, the forward coefficient differences in those directions vanish as
well. The forward estimate first controls $\widehat x$ by the sources and these
combinations; the backward estimate then controls $\widehat y$ by $\widehat x$
and its sources. Lemma~\ref{lem:apriori} makes this closure explicit. These
algebraic perturbations are used to verify coefficient conditions; projected
solution variables additionally satisfy their defining conditional relations.

\subsection{One-sided stability estimates}
The following standard recursive estimates are used with frozen companion variables. Their proofs appear in Section~\ref{sec:proofs}.

For the forward direction, consider the mean-field recursion
\begin{flalign}\label{1}
	&
	\begin{cases}
		x_{k+1}=b(k,x_k,{\mathbb{E}}[x_k])+\sigma(k,x_k,\mathbb{E}[x_k])w_k,\\
		x_0=\eta,\quad k\in \mathbb{T}
	\end{cases}
	&
\end{flalign}

\begin{assumption}\label{ass:forward}
	$\eta\in\mathbb{R}^n$ and $(b,\sigma)$ are two given random mappings:
	\begin{flalign*}
		&
		\begin{aligned}
			b&:\Omega\times \mathbb{T}\times\mathbb{R}^n\times\mathbb{R}^n\to\mathbb{R}^n,\\
			\sigma&:\Omega\times \mathbb{T}\times\mathbb{R}^n\times\mathbb{R}^n\to\mathbb{R}^n,
		\end{aligned}
		&
	\end{flalign*}
	jointly measurable in $(\omega,x_1,x_2)$ and satisfying the following conditions, with all pointwise bounds holding simultaneously on a common set of probability one:
	
	(i) For any $x_1,x_2\in\mathbb{R}^n$, $b(k,x_1,x_2)$ and $\sigma(k,x_1,x_2)$ are $\mathcal{F}_{k-1}$-measurable, $k\in \mathbb{T}$;
	
	(ii) $b(\cdot,0,0)\in L^2_{\mathbb{F}}(\mathbb{T};\mathbb{R}^n)$ and $\sigma(\cdot,0,0)\in L^2_{\mathbb{F}}(\mathbb{T};\mathbb{R}^n)$;
	
	(iii) One constant $L>0$ works for every $k\in\mathbb T$ and all
	$x_1,\bar x_1,x_2,\bar x_2\in\mathbb R^n$ in the bounds
	\[
	\left| b(k,x_1,x_2)-b(k,\bar{x}_1,\bar{x}_2) \right|
	\leq L \left( |x_1-\bar{x}_1|+|x_2-\bar{x}_2| \right),
	\]
	and
	\[
	\left| \sigma(k,x_1,x_2)-\sigma(k,\bar{x}_1,\bar{x}_2) \right|
	\leq L \left( |x_1-\bar{x}_1|+|x_2-\bar{x}_2| \right).
	\]
\end{assumption}

\begin{lemma}\label{lem:forward}
	If Assumption~\ref{ass:forward} holds, the MF-S$\Delta$E~\eqref{1} determined by $(b,\sigma,\eta)$ has exactly one solution $x(\cdot)\in L^2_{\mathbb F}(\overline{\mathbb T};\mathbb R^n)$, and
	\begin{flalign}\label{2}
		&
		\begin{aligned}
			&\mathbb{E}\Bigg[\sum_{k=0}^{N}|x_k|^2\Bigg]\\ &\leq C\mathbb{E} \Big[|\eta|^2 + \sum_{k=0}^{N-1}\left(|b(k,0,0)|^2 +|\sigma(k,0,0)|^2\right) \Big],
		\end{aligned}
		&
	\end{flalign}
	where $C=C(N,L)>0$. For another admissible coefficient triple $(\bar b,\bar\sigma,\bar\eta)$ and its solution $\bar x(\cdot)$, one has
	
	\begin{flalign}\label{3}
		&
		\begin{aligned}
			&\mathbb{E}\left[\sum_{k=0}^{N}|\hat{x}_{k}|^{2}\right]\\
			&\leq C\mathbb{E}\Bigg[ |\eta-\bar{\eta}|^{2}\\
			&\quad+\sum_{k=0}^{N-1}\left|b(k,\bar{x}_{k},\mathbb{E}[\bar{x}_{k}])-\bar{b}(k,\bar{x}_{k},\mathbb{E}[\bar{x}_{k}])\right|^{2} \\
			&\quad+\sum_{k=0}^{N-1}\left|\sigma(k,\bar{x}_{k},\mathbb{E}[\bar{x}_{k}])-\bar{\sigma}(k,\bar{x}_{k},\mathbb{E}[\bar{x}_{k}])\right|^{2}\Bigg],
		\end{aligned}
		&
	\end{flalign}
	with the same dependence of $C$.
\end{lemma}

For the reverse direction, consider
	\begin{flalign}\label{5}
	&
	\begin{cases}
		y_k = f\bigl(k+1, \theta_{k+1}, \mathbb{E}[\theta_{k+1}]\bigr), & \\
		\theta_{k+1}=\bigl(y'_{k+1},z'_{k+1}\bigr), & \\
		y_N = \xi, \quad k \in \mathbb{T}.
	\end{cases}
	&
\end{flalign}
where $z_{k+1} = y_{k+1} w_k$, $y'_{k+1} = \mathbb{E}[y_{k+1} | \mathcal{F}_{k-1}]$,
$z'_{k+1} = \mathbb{E}[y_{k+1} w_k | \mathcal{F}_{k-1}]$. Thus the driver is
written in the same packaged form $f(k+1,\theta,\mathbb E[\theta])$ that is used
for the coupled system~\eqref{11}; here the package is the $2$-dimensional
$\theta_{k+1}=(y'_{k+1},z'_{k+1})$.
Its data $(\xi,f)$ are required to satisfy:

\begin{assumption}\label{ass:backward}
	$\xi \in L^2_{\mathcal{F}_{N-1}}(\Omega, \mathbb{R}^n)$, $f$ is a given mapping:
	$
	f:\Omega \times \{1,2,\ldots,N\}
	\times \mathbb R^{2n}
	\times \mathbb R^{2n}
	\longrightarrow \mathbb R^n$, jointly measurable in its random and Euclidean arguments and satisfying the following conditions, with all pointwise bounds holding simultaneously on a common set of probability one:
	
	(i) For any $\theta, \bar\theta \in \mathbb{R}^{2n}$, $f(k+1, \theta, \bar\theta)$ is $\mathcal{F}_{k-1}$-measurable, $k \in \mathbb{T}$;
	
	(ii) $\{f(k+1,0,0)\}_{k\in\mathbb T}\in L^2_{\mathbb F}(\mathbb T;\mathbb R^n)$;
	
	(iii) There is an $L>0$, uniform in \(k\), for which every
	$\theta,\bar\theta,\widetilde\theta,\widetilde{\bar\theta}\in\mathbb R^{2n}$ satisfies
	\begin{flalign*}
		&
		\begin{aligned}
			& \bigl| f(k+1, \theta, \bar\theta) - f(k+1, \widetilde{\theta}, \widetilde{\bar\theta}) \bigr| \\
			& \le L\bigl( |\theta - \widetilde{\theta}| + |\bar\theta - \widetilde{\bar\theta}| \bigr).
		\end{aligned}
		&
	\end{flalign*}
\end{assumption}

\begin{lemma}\label{lem:backward}
	Under Assumption~\ref{ass:backward}, equation~\eqref{5} has a unique $y(\cdot)\in L^2_{\mathbb F}(\overline{\mathbb T};\mathbb R^n)$ satisfying
	\begin{flalign}\label{6}
		&
		\begin{aligned}
			&\mathbb{E}\left[ \sum_{k=0}^N |y_k|^2 \right]\\
			&\le C\mathbb{E}\left[ |\xi|^2 + \sum_{k=0}^{N-1} |f(k+1, 0,0)|^2 \right].
		\end{aligned}
		&
	\end{flalign}
	
	If another admissible pair $(\bar\xi,\bar f)$ produces $\bar y(\cdot)$, then
	
	\begin{flalign}\label{7}
		&
		\begin{aligned}
			&\mathbb{E}\left[ \sum_{k=0}^N |\hat{y}_k|^2 \right]\\
			&\le C\mathbb{E}\Bigg[|\xi - \bar{\xi}|^2\\
			&\quad+ \sum_{k=0}^{N-1} \Big| f\bigl(k+1, \bar\theta_{k+1}, \mathbb{E}[\bar\theta_{k+1}]\bigr) \\
			&\quad- \bar{f}\bigl(k+1, \bar\theta_{k+1}, \mathbb{E}[\bar\theta_{k+1}]\bigr) \Big|^2 \Bigg].
		\end{aligned}
		&
	\end{flalign}
	In both displays, $C=C(N,L)>0$, and $\bar\theta_{k+1}=(\bar y'_{k+1},\bar z'_{k+1})$.
\end{lemma}

The dependencies are now explicit: Lemmas~\ref{lem:forward} and
\ref{lem:backward} supply the decoupled bounds, the continuation family
preserves the domination--monotonicity constants, and a uniform contraction
step completes Theorem~\ref{thm:main}. The details follow.

\section{Proofs and Continuation Construction}\label{sec:proofs}
\subsection{Proofs of the one-sided estimates}
\begin{proof}[Proof of Lemma~\ref{lem:forward}]
Chronological recursion constructs the solution. Predictability and conditional
unit variance give
\[
\mathbb E|b_k+\sigma_kw_k|^2=\mathbb E|b_k|^2+\mathbb E|\sigma_k|^2,
\]
since the mixed term has zero conditional expectation. Global Lipschitz
continuity ensures square integrability at each finite step.
For two data sets, let
\begin{align*}
\beta_k&=b(k,\bar x_k,\mathbb E\bar x_k)-\bar b(k,\bar x_k,\mathbb E\bar x_k),\\
\chi_k&=\sigma(k,\bar x_k,\mathbb E\bar x_k)-\bar\sigma(k,\bar x_k,\mathbb E\bar x_k).
\end{align*}
The same identity, the Lipschitz bounds, and
$|\mathbb E\widehat x_k|^2\le\mathbb E|\widehat x_k|^2$ yield
\[
\mathbb E|\widehat x_{k+1}|^2
\le C_L\mathbb E|\widehat x_k|^2
+C_L\mathbb E(|\beta_k|^2+|\chi_k|^2).
\]
Starting from $\widehat x_0=\eta-\bar\eta$, finite iteration and summation
give~\eqref{3} with $C(N,L)$. Identical data imply uniqueness; comparison
with the zero recursion gives~\eqref{2}.
\end{proof}

\begin{proof}[Proof of Lemma~\ref{lem:backward}]
For $Y\in L^2_{\mathcal F_k}$, conditional Jensen and Cauchy--Schwarz give
\begin{equation}\label{eq:projection-bound}
\begin{aligned}
\mathbb E|\mathbb E[Y\mid\mathcal F_{k-1}]|^2&\le\mathbb E|Y|^2,\\
\mathbb E|\mathbb E[Yw_k\mid\mathcal F_{k-1}]|^2&\le\mathbb E|Y|^2.
\end{aligned}
\end{equation}
In the second inequality use $\mathbb E[w_k^2\mid\mathcal F_{k-1}]=1$.
In fact, writing $P_kY=\mathbb E[Y\mid\mathcal F_{k-1}]$ and
$Q_kY=\mathbb E[Yw_k\mid\mathcal F_{k-1}]$, the residual
$R=Y-P_kY-(Q_kY)w_k$ belongs to $L^2$ and satisfies
\[
\mathbb E|Y|^2=\mathbb E|P_kY|^2+\mathbb E|Q_kY|^2+\mathbb E|R|^2.
\]
Indeed, conditional centering and unit variance make $P_kY$ and $(Q_kY)w_k$
orthogonal, and the defining projections make $R$ orthogonal to both.
For unbounded predictable factors, these identities follow first after
truncation and then by $L^2$ convergence. Thus both projection operators are
bounded, including when $Yw_k$ itself is only integrable.
Thus, starting from $y_N=\xi$, backward evaluation gives a predictable
square-integrable solution. Write $e_k$ for the coefficient difference evaluated
at $(\bar y'_{k+1},\mathbb E\bar y'_{k+1},\bar z'_{k+1},\mathbb E\bar z'_{k+1})$,
as in~\eqref{7}. Lipschitz continuity,~\eqref{eq:projection-bound}, and Jensen
give
\[
\mathbb E|\widehat y_k|^2
\le C_L\mathbb E|\widehat y_{k+1}|^2+C_L\mathbb E|e_k|^2.
\]
Iteration from $\widehat y_N=\xi-\bar\xi$ and summation prove~\eqref{7}.
Identical data imply uniqueness, and comparison with $(\bar f,\bar\xi)=(0,0)$
proves~\eqref{6}.
\end{proof}

\subsection{Continuation family and structural preservation}
We now build the homotopy required by the already stated principal theorem.
The first monotonicity orientation is handled directly; the second is recovered
through Remark~\ref{rem:sign}.

For any $(\xi, \eta, \rho(\cdot)) \in \mathcal{H}(\overline{\mathbb{T}})$ with $\rho(\cdot)=(\varphi(\cdot),\psi(\cdot),\gamma(\cdot))$, introduce the family of MF-FBS$\Delta$Es indexed by $\alpha\in[0,1]$:
\begin{flalign}\label{16}
	&
	\begin{cases}
		\begin{aligned}
			x_{k+1}^{\alpha} = &\left[ b^{\alpha}(k, \theta_k^{\alpha}, \mathbb{E}[\theta_k^{\alpha}]) + \psi_k \right]
			\\&+ \left[ \sigma^{\alpha}(k, \theta_k^{\alpha}, \mathbb{E}[\theta_k^{\alpha}]) + \gamma_k \right] w_k,
		\end{aligned} \\
		y_k^{\alpha} = - \left[ f^{\alpha}(k+1, \theta_k^{\alpha}, \mathbb{E}[\theta_k^{\alpha}]) + \varphi_k \right], \\
		x_0^{\alpha} = \Lambda^{\alpha}(y_0^{\alpha}) + \xi, \\
		y_N^{\alpha} = \Phi^{\alpha}(x_N^{\alpha}, \mathbb{E}[x_N^{\alpha}]) + \eta,
	\end{cases}
	&
\end{flalign}
where
\begin{flalign*}
	&
	\begin{aligned}
			\theta_k^{\alpha}
			&= (x_k^{\alpha}, y_{k+1}'^{\alpha}, z_{k+1}'^{\alpha}) \\
			&= \left( x_k^{\alpha},\ \mathbb{E}\left[ y_{k+1}^{\alpha} \big| \mathcal{F}_{k-1} \right],\ \mathbb{E}\left[ y_{k+1}^{\alpha} w_k \big| \mathcal{F}_{k-1} \right] \right).
	\end{aligned}
	&
\end{flalign*}

For pointwise arguments
\begin{flalign*}
	&
	\theta=(x,y,z),\qquad
	\bar{\theta}=(\bar{x},\bar{y},\bar{z})
	\in \mathbb{R}^{3n},
	&
\end{flalign*}
define the coefficient family
\begin{flalign}\label{17}
	&
	\begin{cases}
		\begin{aligned}
			\Phi^{\alpha}(x,\bar{x})
			={}& \alpha\Phi(x,\bar{x})
			\\
			&+ (1-\alpha)\nu\Big\{
			G^TG(x-\bar{x})
			+\widetilde{G}^T\widetilde{G}\bar{x}
			\Big\},
		\end{aligned}
		\\[1em]
		\begin{aligned}
			\Lambda^{\alpha}(y)
			={}& \alpha\Lambda(y)\\
			&-(1-\alpha)\mu M^TMy,
		\end{aligned}
		\\[1em]
		\begin{aligned}
			f^{\alpha}(k+1,\theta,\bar{\theta})
			={}& \alpha f(k+1,\theta,\bar{\theta})
			\\
			&-(1-\alpha)\nu\Big\{
			A_k^TA_k(x-\bar{x})
			+\widetilde{A}_k^T\widetilde{A}_k\bar{x}
			\Big\},
		\end{aligned}
		\\[1em]
		\begin{aligned}
			&b^{\alpha}(k,\theta,\bar{\theta})
			=\alpha b(k,\theta,\bar{\theta})
			\\
			&\quad-(1-\alpha)\mu\Bigg\{
			B_k^T\Big[
			B_k(y-\bar{y})+C_k(z-\bar{z})
			\Big]
			\\
			&\qquad\qquad
			+\widetilde{B}_k^T\Big(
			\widetilde{B}_k\bar{y}
			+\widetilde{C}_k\bar{z}
			\Big)
			\Bigg\},
		\end{aligned}
		\\[0.3em]
		\begin{aligned}
			&\sigma^{\alpha}(k,\theta,\bar{\theta})
			=\alpha \sigma(k,\theta,\bar{\theta})
			\\
			&\quad-(1-\alpha)\mu\Bigg\{
			C_k^T\Big[
			B_k(y-\bar{y})+C_k(z-\bar{z})
			\Big]
			\\
			&\qquad\qquad
			+\widetilde{C}_k^T\Big(
			\widetilde{B}_k\bar{y}
			+\widetilde{C}_k\bar{z}
			\Big)
			\Bigg\}.
		\end{aligned}
	\end{cases}
	&
\end{flalign}
Here \(k\in\mathbb{T}\), \(\alpha\in[0,1]\), and the coefficients are
evaluated along the stochastic system by taking
\begin{flalign*}
	&
	(\theta,\bar{\theta})
	=
	\bigl(\theta_k,\mathbb{E}[\theta_k]\bigr).
	&
\end{flalign*}

Thus the pointwise definitions above reproduce the centered/mean
form of the continuation system~\eqref{16}.

Also write
\begin{flalign*}
	&
	\begin{aligned}
		\Gamma^\alpha(k,\theta,\bar{\theta})
		:={}&\bigl(
		f^\alpha(k+1,\theta,\bar{\theta}),
		b^\alpha(k,\theta,\bar{\theta}),
		\sigma^\alpha(k,\theta,\bar{\theta})
		\bigr).
	\end{aligned}
	&
\end{flalign*}
All matrix norms are operator norms. A generic \(K\), allowed to vary, depends
only on \(N,L,\mu,\nu\) and the deterministic bounds in Assumption~\ref{ass:dm}.

Under the first alternative, enlarge the original Lipschitz constant \(L_0\), if needed, so that
\begin{flalign*}
	&
	L \geq \max\left\{ L_0,\ \mu \kappa_M^2,\ \frac{\nu}{2}\kappa_G^2,\ \frac{\nu}{2}\kappa_A^2,\ \frac{\mu}{2}\left(\kappa_B^2 + \kappa_C^2\right) \right\},
	&
\end{flalign*}
Then the coefficient triple
\begin{flalign*}
	&
	(\Lambda^\alpha, \Phi^\alpha, \Gamma^\alpha)
	&
\end{flalign*}
is covered uniformly by the following lemma.

\begin{lemma}[Preservation of the domination-monotonicity structure]\label{lem:preservation}
	If \((\Lambda,\Phi,\Gamma)\) satisfies
	Assumptions~\ref{ass:lipschitz} and~\ref{ass:dm} in the first monotonicity orientation, let
	\begin{flalign*}
		&
		(\Lambda^\alpha,\Phi^\alpha,\Gamma^\alpha),
		\qquad \alpha\in[0,1],
		&
	\end{flalign*}
	be given by~\eqref{17}. For every \(\alpha\in[0,1]\), this coefficient triple
	satisfies the same assumptions with constants \(\mu,\nu\) and matrices
	\begin{flalign*}
		&
		M,\quad G,\widetilde G,\quad
		A_k,\widetilde A_k,\quad B_k,\widetilde B_k,\quad
		C_k,\widetilde C_k .
		&
	\end{flalign*}
	Its Lipschitz constant is uniform in \(\alpha\).
	
\end{lemma}

\begin{proof}
	Work under the first alternative. The added terms in~\eqref{17} are bounded
	deterministic linear maps of centered and mean variables. Measurability,
	integrability, and a Lipschitz bound uniform in \(\alpha\) follow immediately.
	For the remaining properties, write
	\begin{flalign*}
		&
		X^{(1)}=X-\mathbb E[X],\qquad X^{(2)}=\mathbb E[X].
		&
	\end{flalign*}
	For two triples
	\begin{flalign*}
		&
		\Theta=(X,Y,Z),\qquad \bar\Theta=(\bar X,\bar Y,\bar Z),
		&
	\end{flalign*}
	set
	\begin{flalign*}
		&
		\widehat X=X-\bar X,\quad
		\widehat Y=Y-\bar Y,\quad
		\widehat Z=Z-\bar Z,
		&
	\end{flalign*}
	and define
	\begin{flalign*}
		&
		\mathcal A_k(\widehat X)
		:= A_k\widehat X^{(1)}+\widetilde A_k\widehat X^{(2)},
		&
	\end{flalign*}
	\begin{flalign*}
		&
		\mathcal B_k(\widehat Y,\widehat Z)
		:=B_k\widehat Y^{(1)}+C_k\widehat Z^{(1)}
		+\widetilde B_k\widehat Y^{(2)}
		+\widetilde C_k\widehat Z^{(2)} .
		&
	\end{flalign*}
	Deterministic matrices preserve centered/mean orthogonality; for instance,
	\begin{flalign*}
		&
		\mathbb E\left[
		\left|G\widehat X^{(1)}+\widetilde G\widehat X^{(2)}\right|^2
		\right]
		=
		\mathbb E\left[|G\widehat X^{(1)}|^2
		+
		|\widetilde G\widehat X^{(2)}|^2 \right].
		&			
	\end{flalign*}
	The size condition also gives
	\begin{flalign*}
		&
		|M^TMv|\le \kappa_M|Mv|,
		&
	\end{flalign*}
	and
	\begin{flalign*}
		&
		\begin{aligned}
			&\mathbb E\left|B_k^TU_k+\widetilde B_k^T\widetilde U_k\right|^2
			\le \kappa_B^2\,\mathbb E|U_k+\widetilde U_k|^2,\\
			&\mathbb E\left|C_k^TU_k+\widetilde C_k^T\widetilde U_k\right|^2
			\le \kappa_C^2\,\mathbb E|U_k+\widetilde U_k|^2,
		\end{aligned}
		&
	\end{flalign*}
	for centered \(U_k\) and deterministic \(\widetilde U_k\).
	
	For monotonicity, begin with \(\Lambda^\alpha\):
	\begin{flalign*}
		&
		\begin{aligned}
			&\left\langle
			\Lambda^\alpha(y)-\Lambda^\alpha(\bar y),\,y-\bar y
			\right\rangle  \\
			& =
			\alpha
			\left\langle
			\Lambda(y)-\Lambda(\bar y),\,y-\bar y
			\right\rangle
			-(1-\alpha)\mu |M(y-\bar y)|^2  \\
			& \le
			-\alpha\mu |M(y-\bar y)|^2
			-(1-\alpha)\mu |M(y-\bar y)|^2  \\
			& =
			-\mu |M(y-\bar y)|^2 .
		\end{aligned}
		&
	\end{flalign*}
	The terminal map gives
	\begin{flalign*}
		&
		\begin{aligned}
			&\mathbb E\left[
			\left\langle
			\Phi^\alpha(X,\mathbb E[X])
			-\Phi^\alpha(\bar X,\mathbb E[\bar X]),
			\widehat X
			\right\rangle
			\right]   \\
			& =
			\alpha\mathbb E\left[
			\left\langle
			\Phi(X,\mathbb E[X])
			-\Phi(\bar X,\mathbb E[\bar X]),
			\widehat X
			\right\rangle
			\right]   \\
			&\quad
			+(1-\alpha)\nu
			\mathbb E\left[
			\left\langle
			G^{T}G\widehat X^{(1)}
			+\widetilde G^{T}\widetilde G\widehat X^{(2)},
			\widehat X
			\right\rangle
			\right]  \\
			& \ge
			\nu\mathbb E\left[
			\left|
			G\widehat X^{(1)}+\widetilde G\widehat X^{(2)}
			\right|^2
			\right].
		\end{aligned}
		&
	\end{flalign*}
	The same decomposition applied to \(\Gamma^\alpha=(f^\alpha,b^\alpha,\sigma^\alpha)\) yields
	\begin{flalign*}
		&
		\begin{aligned}
			&\mathbb E\left[
			\left\langle
			\Gamma^\alpha(k,\Theta,\mathbb E[\Theta])
			-\Gamma^\alpha(k,\bar\Theta,\mathbb E[\bar\Theta]),
			\Theta-\bar\Theta
			\right\rangle
			\right]  \\
			& =
			\alpha\mathbb E\left[
			\left\langle
			\Gamma(k,\Theta,\mathbb E[\Theta])
			-\Gamma(k,\bar\Theta,\mathbb E[\bar\Theta]),
			\Theta-\bar\Theta
			\right\rangle
			\right]  \\
			&\quad
			-(1-\alpha)\nu
			\mathbb E\left[
			|\mathcal A_k(\widehat X)|^2
			\right]
			-(1-\alpha)\mu
			\mathbb E\left[
			|\mathcal B_k(\widehat Y,\widehat Z)|^2
			\right]  \\
			& \le
			-\nu\mathbb E\left[
			|\mathcal A_k(\widehat X)|^2
			\right]
			-\mu\mathbb E\left[
			|\mathcal B_k(\widehat Y,\widehat Z)|^2
			\right].
		\end{aligned}
		&
	\end{flalign*}
	This proves preservation of monotonicity. In Case~A, domination and the size bound give
	\begin{flalign*}
		&
		\begin{aligned}
			&|\Lambda^\alpha(y)-\Lambda^\alpha(\bar y)|\\
			&\le
			\alpha |\Lambda(y)-\Lambda(\bar y)|
			+(1-\alpha)\mu |M^TM(y-\bar y)|  \\
			&\le
			\frac{\alpha}{\mu}|M(y-\bar y)|
			+(1-\alpha)\mu\|M\|\,|M(y-\bar y)|  \\
			&\le
			\frac{1}{\mu}|M(y-\bar y)| .
		\end{aligned}
		&
	\end{flalign*}
	With \(X\) fixed, the drift satisfies
	\begin{flalign*}
		&
		\begin{aligned}
			&\left(
			\mathbb E\left[
			\left|
			b^\alpha(k,\Theta,\mathbb E[\Theta])
			-b^\alpha(k,\bar\Theta,\mathbb E[\bar\Theta])
			\right|^2
			\right]
			\right)^{1/2}  \\
			& \le
			\frac{\alpha}{\mu}
			\left(
			\mathbb E\left[
			|\mathcal B_k(\widehat Y,\widehat Z)|^2
			\right]
			\right)^{1/2}\\
			&\quad+(1-\alpha)\mu\kappa_B
			\left(
			\mathbb E\left[
			|\mathcal B_k(\widehat Y,\widehat Z)|^2
			\right]
			\right)^{1/2}  \\
			& \le
			\frac{1}{\mu}
			\left(
			\mathbb E\left[
			|\mathcal B_k(\widehat Y,\widehat Z)|^2
			\right]
			\right)^{1/2}.
		\end{aligned}
		&
	\end{flalign*}
	Squaring and using \(\mu^2\kappa_B\le1\) gives drift domination; replacing
	\(B\) by \(C\) and using \(\mu^2\kappa_C\le1\) treats \(\sigma^\alpha\).
	
	In Case~B the terminal estimate is
	\begin{flalign*}
		&
		\begin{aligned}
			&\left(
			\mathbb E\left[
			\left|
			\Phi^\alpha(X,\mathbb E[X])
			-\Phi^\alpha(\bar X,\mathbb E[\bar X])
			\right|^2
			\right]
			\right)^{1/2}  \\
			& \le
			\frac{1}{\nu}
			\left(
			\mathbb E\left[
			\left|
			G\widehat X^{(1)}+\widetilde G\widehat X^{(2)}
			\right|^2
			\right]
			\right)^{1/2},
		\end{aligned}
		&
	\end{flalign*}
	by \(\nu^2\kappa_G\le1\). With \(Y,Z\) fixed, the driver obeys
	\begin{flalign*}
		&
		\begin{aligned}
			&\left(
			\mathbb E\left[
			\left|
			f^\alpha(k+1,\Theta,\mathbb E[\Theta])
			-f^\alpha(k+1,\bar\Theta,\mathbb E[\bar\Theta])
			\right|^2
			\right]
			\right)^{1/2}  \\
			& \le
			\frac{1}{\nu}
			\left(
			\mathbb E\left[
			|\mathcal A_k(\widehat X)|^2
			\right]
			\right)^{1/2},
		\end{aligned}
		&
	\end{flalign*}
	using \(\nu^2\kappa_A\le1\). Thus all required bounds hold independently of
	\(\alpha\).
\end{proof}

\subsection{Uniform estimate and continuation step}

We henceforth work with the first orientation.

When $\alpha=0$, MF-FBS$\Delta$E~\eqref{16} can be rewritten in the following form:
\begin{flalign}\label{18}
	&
	\begin{cases}
		\begin{aligned}
			x_{k+1}^{0}
			= &-\mu\Big[B_k^T\Big(B_k y_{k+1}'^{0(1)} + C_k z_{k+1}'^{0(1)}\Big)
			\\&+ \widetilde{B}_k^T\Big(\widetilde{B}_k y_{k+1}'^{0(2)} + \widetilde{C}_k z_{k+1}'^{0(2)}\Big)\Big] + \psi_k
			\\&+ \Big\{-\mu\Big[C_k^T\Big(B_k y_{k+1}'^{0(1)} + C_k z_{k+1}'^{0(1)}\Big)
			\\&+ \widetilde{C}_k^T\Big(\widetilde{B}_k y_{k+1}'^{0(2)} + \widetilde{C}_k z_{k+1}'^{0(2)}\Big)\Big] + \gamma_k\Big\}w_k,
		\end{aligned}\\
		y_k^{0}
		= -\Big[-\nu\Big(A_k^TA_k x_k^{0(1)} + \widetilde{A}_k^T\widetilde{A}_k x_k^{0(2)}\Big) + \varphi_k\Big], \\
		x_0^{0}
		= -\mu M^T M y_0^{0} + \xi, \\
		y_N^{0}
		= \nu\Big(G^TG x_N^{0(1)} + \widetilde{G}^T\widetilde{G} x_N^{0(2)}\Big) + \eta.
	\end{cases}
	&
\end{flalign}
Here the superscripts \((1),(2)\) are those of Assumption~\ref{ass:dm}; for example,
$y_{k+1}^{\prime0(1)}=(\mathbb E[y_{k+1}^0\mid\mathcal F_{k-1}])^{(1)}$.

Equation~\eqref{18} is sequentially solvable. In Case~A, Lemma~\ref{lem:backward}
first determines \(y^0\), after which Lemma~\ref{lem:forward} determines \(x^0\).
Case~B reverses this order. The intermediate projected terms are admissible
square-integrable inputs, so the resulting pair is the unique element of
$N^2(\overline{\mathbb T};\mathbb R^{2n})$ solving~\eqref{18}.

At \(\alpha=1\) with zero sources,~\eqref{16} equals~\eqref{11}. A uniform
local step from any solvable \(\alpha_0\) therefore reaches \(1\) after finitely
many iterations, as in~\cite{Hu1995}. The required estimate is next.

\begin{lemma}\label{lem:apriori}
	Work in the first monotonicity orientation at a fixed
	\(\alpha\in[0,1]\). For the two solutions of~\eqref{16} driven,
	respectively, by \((\xi,\eta,\rho)\) and
	\((\bar\xi,\bar\eta,\bar\rho)\), the total difference energy obeys
	\begin{flalign}\label{19}
		&
		\begin{aligned}
			\mathbb{E}\left[ \sum_{k=0}^{N} |\widehat{x}_k|^2 + \sum_{k=0}^{N} |\widehat{y}_k|^2 \right] \leq K \mathbb{E}[\widehat{J}].
		\end{aligned}
		&
	\end{flalign}
	where
	\begin{flalign}\label{20}
		&
		\begin{aligned}
			\widehat{J} = |\widehat{\eta}|^2 + \sum_{k=0}^{N-1} |\widehat{\varphi}_k|^2 + \sum_{k=0}^{N-1} |\widehat{\psi}_k|^2 + \sum_{k=0}^{N-1} |\widehat{\gamma}_k|^2 + |\widehat{\xi}|^2.
		\end{aligned}
		&
	\end{flalign}
	Here $\widehat{\xi} = \xi - \bar{\xi}$, $\widehat{\varphi} = \varphi - \bar{\varphi}$, etc. The constant $K$ depends only on $N$, the common Lipschitz bound, $\mu$, $\nu$, and the deterministic matrix bounds in Assumption~\ref{ass:dm}; in particular, it is independent of $\alpha$. In the active cases this allows dependence on \(\mu^{-1}\) when \(\mu>0\) or on \(\nu^{-1}\) when \(\nu>0\), but on no solution-dependent quantity.
\end{lemma}

\begin{proof}
	Here \(K\) has the dependence stated in the lemma and may vary; only the active
	coercivity parameter is inverted.
	
	\emph{Step 1: forward and backward perturbation estimates.}
	For compactness, define the two mixed triples
	\[
	\theta_k^{,y}:=(\bar x_k,y'_{k+1},z'_{k+1}),
	\qquad
	\theta_k^{,x}:=(x_k,\bar y'_{k+1},\bar z'_{k+1}).
	\]
	Freeze $(y'_{k+1},z'_{k+1})$ when applying Lemma~\ref{lem:forward}; the initial difference is
	\begin{flalign*}
		&
		\alpha(\Lambda(y_0)-\Lambda(\bar y_0))-(1-\alpha)\mu M^TM\hat y_0+\hat\xi,
		&
	\end{flalign*}
	while drift and diffusion are compared at the frozen state $\bar x_k$. Thus

\begin{flalign}
	&
	\mathbb{E}\left[\sum_{k=0}^{N}|\hat x_k|^2\right] \notag\\
	&\le K\mathbb{E}\Bigg[
	\Big|\alpha(\Lambda(y_0)-\Lambda(\bar y_0))-(1-\alpha)\mu M^TM\hat y_0+\hat\xi\Big|^2 \notag\\
	&+\sum_{k=0}^{N-1}\Big|\alpha\big(b(k,\theta_k^{,y},\mathbb{E}[\theta_k^{,y}]) \notag\\
	&\qquad-b(k,\bar\theta_k,\mathbb{E}[\bar\theta_k])\big)
	-(1-\alpha)\mu\Big(B_k^T(B_k\hat y_{k+1}'^{(1)}+C_k\hat z_{k+1}'^{(1)}) \notag\\
	&\qquad+\widetilde B_k^T(\widetilde B_k\hat y_{k+1}'^{(2)}+\widetilde C_k\hat z_{k+1}'^{(2)})\Big)
	+\hat\psi_k\Big|^2 \notag\\
	&+\sum_{k=0}^{N-1}\Big|\alpha\big(\sigma(k,\theta_k^{,y},\mathbb{E}[\theta_k^{,y}]) \notag\\
	&\qquad-\sigma(k,\bar\theta_k,\mathbb{E}[\bar\theta_k])\big)
	-(1-\alpha)\mu\Big(C_k^T(B_k\hat y_{k+1}'^{(1)}+C_k\hat z_{k+1}'^{(1)}) \notag\\
	&\qquad+\widetilde C_k^T(\widetilde B_k\hat y_{k+1}'^{(2)}+\widetilde C_k\hat z_{k+1}'^{(2)})\Big)
	+\hat\gamma_k\Big|^2\Bigg]. \label{21}
	&
\end{flalign}

	Conversely, freeze $(\bar y'_{k+1},\bar z'_{k+1})$ and regard $x_k$ as the parameter in Lemma~\ref{lem:backward}. This gives
\begin{flalign}
	&
	\mathbb{E}\left[\sum_{k=0}^{N}|\hat y_k|^2\right] \notag\\
	&\le K\mathbb{E}\Bigg[\Big|\alpha(\Phi(x_N,\mathbb{E}[x_N])-\Phi(\bar x_N,\mathbb{E}[\bar x_N])) \notag\\
	&\qquad +(1-\alpha)\nu(G^TG\hat x_N^{(1)}+\widetilde G^T\widetilde G\hat x_N^{(2)})+\hat\eta\Big|^2 \notag\\
	&+\sum_{k=0}^{N-1}\Big|\alpha\big(f(k+1,\theta_k^{,x},\mathbb{E}[\theta_k^{,x}]) \notag\\
	&\qquad-f(k+1,\bar\theta_k,\mathbb{E}[\bar\theta_k])\big) \notag\\
	&\qquad-(1-\alpha)\nu(A_k^TA_k\hat x_k^{(1)}+\widetilde A_k^T\widetilde A_k\hat x_k^{(2)})
	+\hat\varphi_k\Big|^2\Bigg]. \label{22}
	&
\end{flalign}
	These two displays fix the comparison convention used below.
The frozen coefficient maps satisfy the one-sided assumptions: their random
parameters are predictable and square integrable, the zero-argument values
are square integrable by the growth bounds, and the Lipschitz constants in
the remaining arguments are unchanged. The forward initial value is
deterministic up to almost-sure equality. Joint measurability and the common
Lipschitz null set justify evaluation at these random parameters.

	The way in which the active parameter enters is now explicit. Define the
	nonnegative combination energies
	\begin{flalign}\label{eq:combination-energies}
		&
		\begin{aligned}
			Q_A={}&|M\widehat y_0|^2
			+\sum_{k=0}^{N-1}\mathbb E\bigl|\mathcal B_k(\widehat y'_{k+1},\widehat z'_{k+1})\bigr|^2,\\
			Q_B={}&\mathbb E|G\widehat x_N^{(1)}+\widetilde G\widehat x_N^{(2)}|^2
			+\sum_{k=0}^{N-1}\mathbb E\bigl|\mathcal A_k(\widehat x_k)\bigr|^2,
		\end{aligned}
		&
	\end{flalign}
	where $\mathcal A_k(\widehat x_k)=A_k\widehat x_k^{(1)}+\widetilde A_k\widehat x_k^{(2)}$
	and $\mathcal B_k(\widehat y'_{k+1},\widehat z'_{k+1})=
	B_k\widehat y_{k+1}'^{(1)}+C_k\widehat z_{k+1}'^{(1)}
	+\widetilde B_k\widehat y_{k+1}'^{(2)}+\widetilde C_k\widehat z_{k+1}'^{(2)}$.
	In Case~A the frozen coercivity term in~\eqref{21} scales as $(1-\alpha)\mu$,
	while domination bounds each coefficient difference by $\mu^{-1}$ times its
	combination; the size condition~(iv) then collapses the two contributions and
	\eqref{22} closes the backward recursion. Hence
	\begin{flalign}\label{eq:mu-dependence}
		&
		\begin{aligned}
			\mathbb E\sum_{k=0}^{N}|\widehat x_k|^2
			&\le C\Bigl(D+\frac{Q_A}{\mu^{2}}\Bigr),\\
			\mathbb E\sum_{k=0}^{N}|\widehat y_k|^2
			&\le C\Bigl(D+\mathbb E\sum_{k=0}^{N}|\widehat x_k|^2\Bigr),
		\end{aligned}
		&
	\end{flalign}
	with $D=\mathbb E[\widehat J]$, and in Case~B the roles are exchanged and
	$\mu^{-2}$ is replaced by $\nu^{-2}$:
	\begin{flalign}\label{eq:nu-dependence}
		&
		\begin{aligned}
			\mathbb E\sum_{k=0}^{N}|\widehat y_k|^2
			&\le C\Bigl(D+\frac{Q_B}{\nu^{2}}\Bigr),\\
			\mathbb E\sum_{k=0}^{N}|\widehat x_k|^2
			&\le C\Bigl(D+\mathbb E\sum_{k=0}^{N}|\widehat y_k|^2\Bigr).
		\end{aligned}
		&
	\end{flalign}
	The constant $C$ depends only on $N$, the common Lipschitz bound, and the
	deterministic matrix bounds; the active parameter appears only through the
	squared factor multiplying $Q_A$ (respectively $Q_B$). This is the precise
	sense in which the constant $K$ of Lemma~\ref{lem:apriori} depends on $\mu$ and
	$\nu$, and it shows that the dependence enters through $\mu^{-2}$ or
	$\nu^{-2}$ alone. The same statements are used, with exactly this meaning of
	$K$, in Lemma~\ref{lem:continuation} and in the completion of
	Theorem~\ref{thm:main}.

	\emph{Step 2: monotonicity identity.}
	Let $H\in L^2_{\mathcal F_{k-1}}$ and $Y\in L^2_{\mathcal F_k}$. Since $H$ is
$\mathcal F_{k-1}$-measurable and $\mathbb E[w_k^2\mid\mathcal F_{k-1}]=1$, the
tower property gives
\[
\mathbb E|Hw_k|^2=\mathbb E\bigl[H^2\,\mathbb E[w_k^2\mid\mathcal F_{k-1}]\bigr]
=\mathbb E|H|^2 ,
\]
so $Hw_k\in L^2$. Cauchy--Schwarz then makes $\langle Hw_k,Y\rangle$
integrable, and pulling out the $\mathcal F_{k-1}$-measurable factor $H$ is
legitimate:
\[
\mathbb E\langle Hw_k,Y\rangle
=\mathbb E\bigl[\bigl\langle H,\mathbb E[Yw_k\mid\mathcal F_{k-1}]\bigr\rangle\bigr]
=\mathbb E\langle H,Q_kY\rangle .
\]
No boundedness or truncation of $H$ is needed, but the hypothesis
$H\in L^2_{\mathcal F_{k-1}}$ is essential, since the argument uses
$\mathcal F_{k-1}$-measurability of $H$ twice; this is precisely where the
discrete-time conditional projection replaces the continuous-time It\^o pairing.
Note also that $z_{k+1}=y_{k+1}w_k$ is only required to lie in $L^1$: for
$y_{k+1}\in L^2_{\mathcal F_k}$ we have $\mathbb E|y_{k+1}w_k|\le
\|y_{k+1}\|_{L^2}\|w_k\|_{L^2}<\infty$, whereas $\mathbb E|y_{k+1}w_k|^2$ may be
infinite because $y_{k+1}$ is not $\mathcal F_{k-1}$-measurable. The projected
variable $Q_kY=\mathbb E[Yw_k\mid\mathcal F_{k-1}]$, by contrast, lies in $L^2$
with $\mathbb E|Q_kY|^2\le\mathbb E|Y|^2$ by conditional Cauchy--Schwarz. This is
why the coercivity--domination pairing is closed at the level of the projected
products $Q_kY$ and not of the raw products $Yw_k$. The corresponding identity
without $w_k$ uses $P_kY$. Thus no martingale representation property or fourth
moment of the innovations is used. The telescoping identity gives
	\begin{flalign}\label{23}
		&
		\begin{aligned}
		&\mathbb{E}[\langle\hat x_N,\hat y_N\rangle-\langle\hat x_0,\hat y_0\rangle]\\
		&
		=\sum_{k=0}^{N-1}\mathbb{E}[\langle\hat x_{k+1},\hat y_{k+1}\rangle-\langle\hat x_k,\hat y_k\rangle].
		\end{aligned}
		&
	\end{flalign}
	
	The left-hand side can be written as
	\begin{flalign}\label{24}
		&\mathbb{E}\left[\langle \hat{x}_N, \hat{y}_N \rangle - \langle \hat{x}_0, \hat{y}_0 \rangle\right] \nonumber\\
		&= \mathbb{E}\Bigg[\Big\langle \hat{x}_N, \alpha\Big( \Phi(x_N, \mathbb{E}[x_N]) - \Phi(\bar{x}_N, \mathbb{E}[\bar{x}_N]) \Big) \nonumber\\
		&\quad + (1-\alpha)\nu\Big( G^T G \hat{x}_N^{(1)} + \widetilde{G}^T \widetilde{G} \hat{x}_N^{(2)} \Big) + \hat{\eta} \Big\rangle \nonumber\\
		&\quad - \Big\langle \alpha\Big( \Lambda(y_0) - \Lambda(\bar{y}_0) \Big) - (1-\alpha)\mu M^T M \hat{y}_0 \nonumber\\
		&\quad + \hat{\xi}, \hat{y}_0 \Big\rangle\Bigg] \displaybreak[3]\nonumber\\
		&= \alpha\mathbb{E}\Bigg[\Big\langle \hat{x}_N, \Phi(x_N, \mathbb{E}[x_N]) - \Phi(\bar{x}_N, \mathbb{E}[\bar{x}_N]) \Big\rangle\Bigg] \nonumber\\
		&\quad + (1-\alpha)\nu\mathbb{E}\Bigg[\Big\langle \hat{x}_N, G^T G \hat{x}_N^{(1)} + \widetilde{G}^T \widetilde{G} \hat{x}_N^{(2)} \Big\rangle\Bigg] \nonumber\\
		&\quad + \mathbb{E}\Bigg[\langle \hat{x}_N, \hat{\eta} \rangle\Bigg] - \mathbb{E}\Bigg[\Big\langle \alpha\Big( \Lambda(y_0) - \Lambda(\bar{y}_0) \Big) \nonumber\\
		&\quad - (1-\alpha)\mu M^T M \hat{y}_0 + \hat{\xi}, \hat{y}_0 \Big\rangle\Bigg].&
	\end{flalign}
	
	Similarly, the right-hand side equals
	\begin{flalign}\label{25}
		&
		\sum_{k=0}^{N-1} \mathbb{E}\left[ \langle \hat{x}_{k+1}, \hat{y}_{k+1} \rangle - \langle \hat{x}_k, \hat{y}_k \rangle \right] \nonumber \\
		&= \sum_{k=0}^{N-1} \mathbb{E}\Bigg[ \mathbb{E}\Big[\Big\langle \hat{b}_k^\alpha + \hat{\psi}_k + (\hat{\sigma}_k^\alpha + \hat{\gamma}_k) w_k, \hat{y}_{k+1} \Big\rangle \Big| \mathcal{F}_{k-1} \Big] \nonumber \\
		&\quad- \langle \hat{x}_k, \hat{y}_k \rangle \Bigg] \nonumber \\[0.5em]
		&= \sum_{k=0}^{N-1} \mathbb{E}\Bigg[ \langle \hat{b}_k^\alpha, \hat{y}_{k+1}' \rangle + \langle \hat{\psi}_k, \hat{y}_{k+1}' \rangle + \langle \hat{\sigma}_k^\alpha, \hat{z}_{k+1}' \rangle \nonumber \\
		&\quad+ \langle \hat{\gamma}_k, \hat{z}_{k+1}' \rangle + \langle \hat{f}_{k+1}^\alpha, \hat{x}_k \rangle + \langle \hat{x}_k, \hat{\varphi}_k \rangle \Bigg] \nonumber \\[0.5em]
		&= \sum_{k=0}^{N-1} \mathbb{E}\Bigg[
		\Big\langle \alpha \hat{b}_k - (1-\alpha)\mu \Big[ B_k^T\big( B_k \hat{y}_{k+1}'^{(1)} + C_k \hat{z}_{k+1}'^{(1)} \big) \nonumber \\
		&\quad+ \widetilde{B}_k^T \Big( \widetilde{B}_k \hat{y}_{k+1}'^{(2)} + \widetilde{C}_k \hat{z}_{k+1}'^{(2)} \Big) \Big], \hat{y}_{k+1}' \Big\rangle \nonumber \\
		&\quad+ \left\langle \alpha \hat{\sigma}_k - (1-\alpha)\mu \Big[ C_k^T\big( B_k \hat{y}_{k+1}'^{(1)} + C_k \hat{z}_{k+1}'^{(1)} \big) \right. \nonumber \\
		&\quad+ \widetilde{C}_k^T\big( \widetilde{B}_k \hat{y}_{k+1}'^{(2)} + \widetilde{C}_k \hat{z}_{k+1}'^{(2)} \big) \Big], \hat{z}_{k+1}' \Big\rangle \nonumber \\
		&\quad+ \left\langle \alpha \hat{f}_{k+1} - (1-\alpha)\nu \big( A_k^T A_k \hat{x}_k^{(1)} + \widetilde{A}_k^T \widetilde{A}_k \hat{x}_k^{(2)} \big), \hat{x}_k \right\rangle \nonumber \\
		&\quad+ \langle \hat{\psi}_k, \hat{y}_{k+1}' \rangle + \langle \hat{\gamma}_k, \hat{z}_{k+1}' \rangle + \langle \hat{x}_k, \hat{\varphi}_k \rangle \Bigg] \displaybreak[3]\nonumber \\[0.5em]
		&= \sum_{k=0}^{N-1} \mathbb{E}\Bigg[
		\langle \alpha \hat{f}_{k+1}, \hat{x}_k \rangle + \langle \alpha \hat{b}_k, \hat{y}_{k+1}' \rangle + \langle \alpha \hat{\sigma}_k, \hat{z}_{k+1}' \rangle \nonumber \\
		&\quad- (1-\alpha)\Bigg( \nu\left\langle A_k^T A_k \hat{x}_k^{(1)} + \widetilde{A}_k^T \widetilde{A}_k \hat{x}_k^{(2)}, \hat{x}_k \right\rangle \nonumber \\
		&\quad+ \mu\left\langle B_k^T\big( B_k \hat{y}_{k+1}'^{(1)} + C_k \hat{z}_{k+1}'^{(1)} \big) \right. \nonumber \\
		&\quad+ \widetilde{B}_k^T\big( \widetilde{B}_k \hat{y}_{k+1}'^{(2)} + \widetilde{C}_k \hat{z}_{k+1}'^{(2)} \big), \hat{y}_{k+1}' \Big\rangle \nonumber \\
		&\quad+ \mu\left\langle C_k^T\big( B_k \hat{y}_{k+1}'^{(1)} + C_k \hat{z}_{k+1}'^{(1)} \big) \right. \nonumber \\
		&\quad+ \widetilde{C}_k^T\big( \widetilde{B}_k \hat{y}_{k+1}'^{(2)} + \widetilde{C}_k \hat{z}_{k+1}'^{(2)} \big), \hat{z}_{k+1}' \Big\rangle\Bigg) \nonumber \\
		&\quad+ \langle \hat{\psi}_k, \hat{y}_{k+1}' \rangle + \langle \hat{\gamma}_k, \hat{z}_{k+1}' \rangle + \langle \hat{x}_k, \hat{\varphi}_k \rangle \Bigg] \displaybreak[3] \nonumber \\[0.5em]
		&= \sum_{k=0}^{N-1} \alpha \mathbb{E}\Bigg[ \big\langle \Gamma(k,\theta,\mathbb{E}[\theta])
		- \Gamma(k,\bar{\theta},\mathbb{E}[\bar{\theta}]), \hat{\theta} \big\rangle \Bigg] \nonumber \\
		&\quad- \sum_{k=0}^{N-1} (1-\alpha)\mathbb{E}\Bigg[ \nu\Big\langle A_k^T A_k \hat{x}_k^{(1)}
		+ \widetilde{A}_k^T \widetilde{A}_k \hat{x}_k^{(2)}, \hat{x}_k \Big\rangle \nonumber \\
		&\quad+ \mu\Big\langle B_k^T\big(B_k \hat{y}_{k+1}'^{(1)} + C_k \hat{z}_{k+1}'^{(1)}\big) \nonumber \\
		&\quad+ \widetilde{B}_k^T\big(\widetilde{B}_k \hat{y}_{k+1}'^{(2)}
		+ \widetilde{C}_k \hat{z}_{k+1}'^{(2)}\big), \hat{y}_{k+1}' \Big\rangle \nonumber \\
		&\quad+ \mu\Big\langle C_k^T\big(B_k \hat{y}_{k+1}'^{(1)} + C_k \hat{z}_{k+1}'^{(1)}\big) \nonumber \\
		&\quad+ \widetilde{C}_k^T\big(\widetilde{B}_k \hat{y}_{k+1}'^{(2)}
		+ \widetilde{C}_k \hat{z}_{k+1}'^{(2)}\big), \hat{z}_{k+1}' \Big\rangle\Bigg] \nonumber \\
		&\quad+ \sum_{k=0}^{N-1} \mathbb{E}\Bigg[ \langle \hat{\psi}_k, \hat{y}_{k+1}' \rangle
		+ \langle \hat{\gamma}_k, \hat{z}_{k+1}' \rangle
		+ \langle \hat{\varphi}_k, \hat{x}_k \rangle \Bigg],
		&
	\end{flalign}
	where, more precisely,
		\begin{flalign*}
			&
			\begin{aligned}
				\widehat b_k
				&=b(k,\theta_k,\mathbb E[\theta_k])
				-b(k,\bar\theta_k,\mathbb E[\bar\theta_k]),\\
				\widehat \sigma_k
				&=\sigma(k,\theta_k,\mathbb E[\theta_k])
				-\sigma(k,\bar\theta_k,\mathbb E[\bar\theta_k]),\\
				\widehat f_{k+1}
				&=f(k+1,\theta_k,\mathbb E[\theta_k])
				-f(k+1,\bar\theta_k,\mathbb E[\bar\theta_k]),
			\end{aligned}
			&
		\end{flalign*}
		and \(\widehat\Gamma_k=(\widehat f_{k+1},\widehat b_k,\widehat\sigma_k)\). Substituting (\ref{24}) and (\ref{25}) into the identity (\ref{23}), and rearranging terms, we obtain:
	\begin{flalign}\label{26}
		&
		\begin{aligned}
		&\alpha\mathbb{E}\langle\hat x_N,\Phi(x_N,\mathbb{E}[x_N])-\Phi(\bar x_N,\mathbb{E}[\bar x_N])\rangle\\
		&\quad+(1-\alpha)\nu\mathbb{E}\langle\hat x_N,G^TG\hat x_N^{(1)}+\widetilde G^T\widetilde G\hat x_N^{(2)}\rangle\\
		&=\alpha\mathbb{E}\langle\Lambda(y_0)-\Lambda(\bar y_0),\hat y_0\rangle
		-(1-\alpha)\mu\mathbb{E}|M\hat y_0|^2\\
		&\quad+\mathbb{E}\langle\hat\xi,\hat y_0\rangle-\mathbb{E}\langle\hat x_N,\hat\eta\rangle\\
		&\quad+\sum_{k=0}^{N-1}\Bigg\{
		\alpha\mathbb{E}\langle\Gamma(k,\theta_k,\mathbb{E}[\theta_k])-\Gamma(k,\bar\theta_k,\mathbb{E}[\bar\theta_k]),\hat\theta_k\rangle\\
		&\quad-(1-\alpha)\nu\mathbb{E}\langle A_k^TA_k\hat x_k^{(1)}+\widetilde A_k^T\widetilde A_k\hat x_k^{(2)},\hat x_k\rangle\\
		&\quad-(1-\alpha)\mu\mathbb{E}\langle B_k^T(B_k\hat y_{k+1}'^{(1)}+C_k\hat z_{k+1}'^{(1)})\\
		&\quad+\widetilde B_k^T(\widetilde B_k\hat y_{k+1}'^{(2)}+\widetilde C_k\hat z_{k+1}'^{(2)}),\hat y_{k+1}'\rangle\\
		&\quad-(1-\alpha)\mu\mathbb{E}\langle C_k^T(B_k\hat y_{k+1}'^{(1)}+C_k\hat z_{k+1}'^{(1)})\\
		&\quad+\widetilde C_k^T(\widetilde B_k\hat y_{k+1}'^{(2)}+\widetilde C_k\hat z_{k+1}'^{(2)}),\hat z_{k+1}'\rangle\\
		&\quad+\mathbb{E}\big[\langle\hat\psi_k,\hat y_{k+1}'\rangle+\langle\hat\gamma_k,\hat z_{k+1}'\rangle+\langle\hat\varphi_k,\hat x_k\rangle\big]\Bigg\}.
		\end{aligned}
		&
	\end{flalign}
	Determinism of the domination matrices gives, for example,
	\begin{flalign}\label{27}
		&
		\begin{aligned}
		&\mathbb{E}\langle\hat x_N,G^TG\hat x_N^{(1)}+\widetilde G^T\widetilde G\hat x_N^{(2)}\rangle\\
		&=\mathbb{E}|G\hat x_N^{(1)}|^2+\mathbb{E}|\widetilde G\hat x_N^{(2)}|^2\\
		&=\mathbb{E}|G\hat x_N^{(1)}+\widetilde G\hat x_N^{(2)}|^2,
		\end{aligned}
		&
	\end{flalign}
	and, for every $k$,
	\begin{flalign}\label{28}
		&
		\begin{aligned}
		&\mathbb{E}\langle A_k^TA_k\hat x_k^{(1)}+\widetilde A_k^T\widetilde A_k\hat x_k^{(2)},\hat x_k\rangle\\
		&=\mathbb{E}|A_k\hat x_k^{(1)}|^2+\mathbb{E}|\widetilde A_k\hat x_k^{(2)}|^2\\
		&=\mathbb{E}|A_k\hat x_k^{(1)}+\widetilde A_k\hat x_k^{(2)}|^2,
		\end{aligned}
		&
	\end{flalign}
	\begin{flalign}\label{29}
		&
		\begin{aligned}
		&\mathbb{E}\Big[\langle B_k^T U_k+\widetilde B_k^T\widetilde U_k,\hat y_{k+1}'\rangle
		+\langle C_k^T U_k+\widetilde C_k^T\widetilde U_k,\hat z_{k+1}'\rangle\Big]\\
		&=\mathbb{E}|U_k|^2+\mathbb{E}|\widetilde U_k|^2
		=\mathbb{E}|U_k+\widetilde U_k|^2,
		\end{aligned}
		&
	\end{flalign}
	where $U_k:=B_k\hat y_{k+1}'^{(1)}+C_k\hat z_{k+1}'^{(1)}$ and $\widetilde U_k:=\widetilde B_k\hat y_{k+1}'^{(2)}+\widetilde C_k\hat z_{k+1}'^{(2)}$. The mixed products vanish because $U_k$ is centered while $\widetilde U_k$ is deterministic.
	
	Insert~\eqref{27}--\eqref{29} into~\eqref{26} and use the first monotonicity alternative:
	\begin{flalign}\label{30}
		&
		\begin{aligned}
		&\mathbb{E}\Bigg[\nu|G\hat x_N^{(1)}+\widetilde G\hat x_N^{(2)}|^2
		+\nu\sum_{k=0}^{N-1}|A_k\hat x_k^{(1)}+\widetilde A_k\hat x_k^{(2)}|^2\\
		&\quad+\mu|M\hat y_0|^2+\mu\sum_{k=0}^{N-1}|B_k\hat y_{k+1}'^{(1)}+C_k\hat z_{k+1}'^{(1)}+\widetilde B_k\hat y_{k+1}'^{(2)}\\
		&\quad+\widetilde C_k\hat z_{k+1}'^{(2)}|^2\Bigg]\\
		&\leq\mathbb{E}\Bigg[-\langle\hat x_N,\hat\eta\rangle+\langle\hat\xi,\hat y_0\rangle
		+\sum_{k=0}^{N-1}\big(\langle\hat\psi_k,\hat y_{k+1}'\rangle+\langle\hat\gamma_k,\hat z_{k+1}'\rangle\\
		&\quad+\langle\hat\varphi_k,\hat x_k\rangle\big)\Bigg].
		\end{aligned}
		&
	\end{flalign}
	
\emph{Step 3: closure and absorption.}
Put
\[
S=\mathbb E\sum_{k=0}^{N}(|\widehat x_k|^2+|\widehat y_k|^2),
\qquad D=\mathbb E\widehat J,
\]
and recall the nonnegative combination energies $Q_A,Q_B$ defined
in~\eqref{eq:combination-energies}.
The initial term is deterministic up to almost-sure equality. By
\eqref{eq:projection-bound}, the squared norm of the solution vector paired
with the sources on the right of~\eqref{30} is bounded by
\begin{align*}
&\mathbb E|\widehat x_N|^2+|\widehat y_0|^2\\
&\quad+\sum_{k=0}^{N-1}\mathbb E\bigl(|\widehat x_k|^2
+|\widehat y'_{k+1}|^2+|\widehat z'_{k+1}|^2\bigr)\le2S.
\end{align*}
Cauchy--Schwarz on the product space of sources therefore gives
\begin{equation}\label{eq:source-pairing}
\mu Q_A+\nu Q_B\le\sqrt{2DS}.
\end{equation}
This bound uses no inverse matrix or full-rank assumption.

In Case~A, the preserved domination bounds in~\eqref{21} and the Lipschitz
bounds in~\eqref{22} yield \eqref{eq:mu-dependence}, that is,
\begin{align*}
\mathbb E\sum_{k=0}^{N}|\widehat x_k|^2&\le C\Bigl(D+\frac{Q_A}{\mu^2}\Bigr),\\
\mathbb E\sum_{k=0}^{N}|\widehat y_k|^2&\le C\left(D+
\mathbb E\sum_{k=0}^{N}|\widehat x_k|^2\right).
\end{align*}
Consequently $S\le C_A(D+Q_A/\mu^2)$ for a fixed $C_A$ independent of $\alpha$.
In Case~B the roles reverse and \eqref{eq:nu-dependence} gives $S\le
C_B(D+Q_B/\nu^2)$. The constants depend only on the structural bounds stated in
the lemma. Kernel directions do not enter the dominated coefficient differences
and are controlled by the companion estimate.

In either case let $c$ be the active parameter, $Q$ the corresponding energy,
and $C_*$ the fixed closure constant. Then~\eqref{eq:source-pairing} gives
\begin{align*}
S&\le C_*D+\frac{\sqrt2C_*}{c}\sqrt{DS}\\
 &\le C_*D+\tfrac12 S+\frac{C_*^2}{c^2}D.
\end{align*}
Thus $S\le(2C_*+2C_*^2/c^2)D$, proving~\eqref{19}. This explicitly identifies
the absorption and its independence of $\alpha$; it does not give a bound
uniform as $c\downarrow0$.
\end{proof}

\begin{lemma}\label{lem:continuation}
	Under the first alternative, suppose~\eqref{16} is uniquely solvable at some
	\(\alpha_0<1\) for every source in \(\mathcal H(\overline{\mathbb T})\).
	There is a step \(\delta_0>0\), independent of \(\alpha_0\), for which the same
	is true at \(\alpha_0+\delta\le1\) whenever \(0<\delta\le\delta_0\).
\end{lemma}

\begin{proof}
	Let $\delta_0 > 0$ be determined below. For any $(x(\cdot), y(\cdot)) \in N^2(\overline{\mathbb{T}}; \mathbb{R}^{2n})$, $\theta(k) = (x_k, y'_{k+1}, z'_{k+1})$ and $\Theta(k) = (X_k, Y'_{k+1},$\\$ Z'_{k+1})$, we introduce the following MF-FBS$\Delta$E with unknown $(X(\cdot), Y(\cdot)) \in N^2(\overline{\mathbb{T}}; \mathbb{R}^{2n})$:
	
	\begin{flalign}\label{39}
		&
		\left\{
		\begin{aligned}
			X_{k+1} &=  \Bigg\{
			-(1-\alpha_0)\mu \Big[
			B_k^\top \big(B_k  Y_{k+1}'^{(1)} + C_k  Z_{k+1}'^{(1)}\big) \\
			& \quad + \widetilde{B}_k^\top \big(\widetilde{B}_k Y_{k+1}'^{(2)} + \widetilde{C}_k Z_{k+1}'^{(2)}\big)
			\Big] \\
			& \quad + \alpha_0 b(k, \Theta_k, \mathbb{E}[\Theta_k]) + \widetilde{\psi}_k
			\Bigg\} \\
			& \quad+ \Bigg\{
			-(1-\alpha_0)\mu \Big[
			C_k^\top \big(B_k Y_{k+1}'^{(1)} + C_k Z_{k+1}'^{(1)}\big) \\
			& \quad + \widetilde{C}_k^\top \big(\widetilde{B}_k  Y_{k+1}'^{(2)} + \widetilde{C}_k Z_{k+1}'^{(2)}\big)
			\Big] \\
			& \quad + \alpha_0 \sigma(k, \Theta_k, \mathbb{E}[\Theta_k]) + \widetilde{\gamma}_k
			\Bigg\} w_k, \\
			Y_k = {} & -\Bigg\{
			-(1-\alpha_0)\nu \Big[
			A_k^\top A_k X_k^{(1)} + \widetilde{A}_k^\top \widetilde{A}_k X_k^{(2)}
			\Big] \\
			& + \alpha_0 f(k+1, \Theta_k, \mathbb{E}[\Theta_k]) + \widetilde{\varphi}_k
			\Bigg\}, \\
			X_0 = {} & \alpha_0 \Lambda(Y_0) - (1-\alpha_0)\mu M^\top M Y_0 + \widetilde{\xi}, \\
			Y_N = {} & \alpha_0 \Phi(X_N, \mathbb{E}[X_N]) + (1-\alpha_0)\nu \Big[
			G^\top G X_N^{(1)} \\
			&+ \widetilde{G}^\top \widetilde{G} X_N^{(2)}
			\Big] + \widetilde{\eta}.
		\end{aligned}
		\right.
		&
	\end{flalign}
	where
	\begin{flalign}\label{40}
		&\widetilde{\psi}_k = \psi_k + \delta b(k, \theta_k, \mathbb{E}[\theta_k]) \nonumber\\
		&\quad + \delta \mu \Bigg[B_k^\top \big(B_k y_{k+1}'^{(1)} + C_k z_{k+1}'^{(1)}\big) \nonumber\\
		&\qquad\qquad + \widetilde{B}_k^\top \big(\widetilde{B}_k y_{k+1}'^{(2)} + \widetilde{C}_k z_{k+1}'^{(2)}\big)\Bigg], \displaybreak[3]\nonumber\\
		&\widetilde{\gamma}_k = \gamma_k + \delta \sigma(k, \theta_k, \mathbb{E}[\theta_k]) \nonumber\\
		&\quad + \delta \mu \Bigg[C_k^\top \big(B_k y_{k+1}'^{(1)} + C_k z_{k+1}'^{(1)}\big) \nonumber\\
		&\qquad\qquad + \widetilde{C}_k^\top \big(\widetilde{B}_k y_{k+1}'^{(2)} + \widetilde{C}_k z_{k+1}'^{(2)}\big)\Bigg], \displaybreak[3]\nonumber\\
		&\widetilde{\varphi}_k = \varphi_k + \delta f(k+1, \theta_k, \mathbb{E}[\theta_k]) \nonumber\\
		&\quad + \delta \nu \Big[A_k^\top A_k x_k^{(1)} + \widetilde{A}_k^\top \widetilde{A}_k x_k^{(2)}\Big], \displaybreak[3]\nonumber\\
		&\widetilde{\eta} = \eta + \delta \Phi(x_N, \mathbb{E}[x_N])
		- \delta \nu \Big[G^\top G x_N^{(1)} + \widetilde{G}^\top \widetilde{G} x_N^{(2)}\Big], \nonumber\\
		&\widetilde{\xi} = \xi + \delta \Lambda(y_0) + \delta \mu M^\top M y_0, \nonumber\\
		&Z_{k+1}=Y_{k+1}w_k,\qquad
		Y_{k+1}'=\mathbb{E}[Y_{k+1}\mid\mathcal F_{k-1}], \nonumber\\
		&Z_{k+1}'=\mathbb{E}[Y_{k+1}w_k\mid\mathcal F_{k-1}].&
	\end{flalign}
	
	Furthermore, we also denote $\widetilde{\rho}(\cdot) = (\widetilde{\varphi}(\cdot), \widetilde{\psi}(\cdot), \widetilde{\gamma}(\cdot))$.
	The Lipschitz and square-integrability assumptions imply that $(\widetilde{\xi}, \widetilde{\eta}, \widetilde{\rho}) \in \mathcal{H}(\overline{\mathbb{T}})$.
	Solvability at \(\alpha_0\) makes~\eqref{39} uniquely solvable:
	$(X(\cdot), Y(\cdot)) \in N^2(\overline{\mathbb{T}}; \mathbb{R}^{2n})$.
	These observations associate with every trial pair $(x,y)$ a unique output
	pair $(X,Y)$. Denote this input--output operator by
	\(\mathcal{T}_{\alpha_0+\delta}:N^2(\overline{\mathbb{T}};\mathbb{R}^{2n})
	\to N^2(\overline{\mathbb{T}};\mathbb{R}^{2n})\), with
	\((X(\cdot),Y(\cdot))
	=\mathcal{T}_{\alpha_0+\delta}((x(\cdot),y(\cdot)))\).
	It remains to choose \(\delta\) so that this map contracts.
	
	Choose two inputs $(x,y)$ and $(\bar x,\bar y)$ in
	$N^2(\overline{\mathbb T};\mathbb R^{2n})$. Write their images as
	$(X,Y)=\mathcal T_{\alpha_0+\delta}(x,y)$ and
	$(\bar X,\bar Y)=\mathcal T_{\alpha_0+\delta}(\bar x,\bar y)$.
	Hats will now mean pairwise subtraction at either level; thus
	$\hat x=x-\bar x$, $\hat y=y-\bar y$, $\hat X=X-\bar X$, and
	$\hat Y=Y-\bar Y$, with the same convention for all derived quantities.
	Lemma~\ref{lem:apriori}, with the time argument suppressed, gives
	\begin{flalign*}
		&
			\mathbb{E}\left[ \sum_{k=0}^{N} |\hat{X}_k|^2
			+ \sum_{k=0}^{N} |\hat{Y}_k|^2 \right] \\
			& \le K \delta^2 \mathbb{E}\Bigg\{
			\Big| \Big( \Phi(x_N, \mathbb{E}[x_N]) - \Phi(\bar{x}_N, \mathbb{E}[\bar{x}_N]) \Big)\\
			&\quad - \nu \big( G^\top G \hat{x}_N^{(1)}
			+ \widetilde{G}^\top \widetilde{G} \hat{x}_N^{(2)} \big) \Big|^2 \\
			& \quad + \big| \Lambda(y_0) - \Lambda(\bar{y}_0)
			+ \mu M^\top M \hat{y}_0 \big|^2 \\
			& \quad + \sum_{k=0}^{N-1} \Big| \Big(
			b(k,\theta_k,\mathbb{E}[\theta_k]) - b(k,\bar{\theta}_k,\mathbb{E}[\bar{\theta}_k]) \Big) \\
			& \quad  + \mu \Big( B_k^\top (B_k \hat{y}_{k+1}'^{(1)} + C_k \hat{z}_{k+1}'^{(1)}) \\
			& \quad  + \widetilde{B}_k^\top (\widetilde{B}_k \hat{y}_{k+1}'^{(2)}
			+ \widetilde{C}_k \hat{z}_{k+1}'^{(2)}) \Big) \Big|^2 \displaybreak[3]\\
			& \quad + \sum_{k=0}^{N-1} \Big| \Big(
			\sigma(k,\theta_k,\mathbb{E}[\theta_k]) - \sigma(k,\bar{\theta}_k,\mathbb{E}[\bar{\theta}_k]) \Big) \\
			& \quad  + \mu \Big( C_k^\top (B_k \hat{y}_{k+1}'^{(1)} + C_k \hat{z}_{k+1}'^{(1)}) \\
			& \quad  + \widetilde{C}_k^\top (\widetilde{B}_k \hat{y}_{k+1}'^{(2)}
			+ \widetilde{C}_k \hat{z}_{k+1}'^{(2)}) \Big) \Big|^2 \\
			& \quad + \sum_{k=0}^{N-1} \Big| \Big(
			f(k{+}1,\theta_k,\mathbb{E}[\theta_k]) - f(k{+}1,\bar{\theta}_k,\mathbb{E}[\bar{\theta}_k]) \\
			& \quad  + \nu \big( A_k^\top A_k \hat{x}_k^{(1)}
			+ \widetilde{A}_k^\top \widetilde{A}_k \hat{x}_k^{(2)} \big) \Big) \Big|^2
			\Bigg\}.
			&
	\end{flalign*}
	Every term on the right is controlled by the common Lipschitz constant $L$
	and the deterministic bounds for
	$G,\widetilde G,M,$ \\ $A_k,\widetilde A_k,B_k,\widetilde B_k,C_k,\widetilde C_k$.
	Consequently one can choose $K'>0$, without dependence on either
	$\alpha_0$ or $\delta$, such that
	\begin{flalign*}
		&
		\left\| (\widehat{X}(\cdot), \widehat{Y}(\cdot)) \right\|_{N^2(\overline{\mathbb{T}}; \mathbb{R}^{2n})}^2
		\le K' \delta^2
		\left\| (\hat{x}(\cdot), \hat{y}(\cdot)) \right\|_{N^2(\overline{\mathbb{T}}; \mathbb{R}^{2n})}^2.
		&
	\end{flalign*}
	This is precisely the product norm of \(N^2(\overline{\mathbb T};\mathbb R^{2n})\).
	
	By selecting \(\delta_0 = 1/(2\sqrt{K'})\), when \(\delta \in (0, \delta_0]\), we obtain \(K'\delta^2\le 1/4\). Equivalently,
			\begin{flalign*}
				&
				\begin{aligned}
					&\|\mathcal T_{\alpha_0+\delta}(x,y)
					-\mathcal T_{\alpha_0+\delta}(\bar x,\bar y)\|_{N^2(\overline{\mathbb T};\mathbb R^{2n})}\\
					&\quad\le \frac12
					\|(x,y)-(\bar x,\bar y)\|_{N^2(\overline{\mathbb T};\mathbb R^{2n})}.
				\end{aligned}
				&
			\end{flalign*}
	Hence the ordinary norm contracts by at most \(1/2\) (equivalently, the
	squared norm by at most \(1/4\)). Banach's theorem supplies exactly one fixed
	point of \(\mathcal T_{\alpha_0+\delta}\); by construction, that fixed point
	is the unique solution of~\eqref{16}.
	
	It remains to verify that $\delta_0$ can be chosen \emph{independent of $\alpha_0$}. Indeed, the constant $K$ in Lemma~\ref{lem:apriori} depends only on $\mathbb{T}$, the Lipschitz constant $L$, $\mu$, $\nu$, and the uniform bounds of the matrices appearing in Assumption~\ref{ass:dm}. By the construction of the $\alpha$-family in (\ref{17}), the coefficients $(\Lambda^\alpha,\Phi^\alpha,\Gamma^\alpha)$ satisfy Assumptions~\ref{ass:lipschitz} and~\ref{ass:dm} with the \emph{same} constants $L,\mu,\nu$ and matrix bounds for all $\alpha\in[0,1]$. Consequently, Lemma~\ref{lem:apriori} yields a constant $K$ that is uniform in $\alpha$, and the derived $K'$ is also independent of $\alpha_0$. Hence $\delta_0$ does not depend on $\alpha_0$, and we can reach $\alpha=1$ from $\alpha=0$ in finitely many steps (at most $\lceil 1/\delta_0 \rceil$ steps). The proof is finished.
\end{proof}

\subsection{Completion of the principal theorem}
The preceding ingredients are now assembled in the order dictated by their
logical dependence.

\begin{proof}
Assume first that the first monotonicity alternative holds. At $\alpha=0$, equation~\eqref{18} is uniquely solvable: in Case~A one first solves the backward equation and then the forward equation, while in Case~B one proceeds in the reverse order. The structural size condition is used in Lemma~\ref{lem:preservation} to keep the domination constants unchanged along the whole continuation family. Lemma~\ref{lem:continuation}, whose step size is uniform in $\alpha$, therefore extends solvability from $\alpha=0$ to $\alpha=1$ in finitely many steps. With zero source terms, equation~\eqref{16} at $\alpha=1$ is exactly~\eqref{11}, proving existence. Uniqueness follows from Lemma~\ref{lem:apriori} with identical source terms.

For the stability estimate, regard the solution associated with $(\bar\Lambda,\bar\Phi,\bar\Gamma)$ as a solution of the $\alpha=1$ equation constructed from $(\Lambda,\Phi,\Gamma)$ with perturbation source terms
\begin{flalign*}
&\begin{aligned}
\xi^\Delta={}&\bar\Lambda(\bar y_0)-\Lambda(\bar y_0),\\
\eta^\Delta={}&\bar\Phi(\bar x_N,\mathbb E[\bar x_N])-\Phi(\bar x_N,\mathbb E[\bar x_N]),\\
\rho^\Delta_k={}&\bar\Gamma(k,\bar\theta_k,\mathbb E[\bar\theta_k])
-\Gamma(k,\bar\theta_k,\mathbb E[\bar\theta_k]).
\end{aligned}&
\end{flalign*}
	Although these differences have the reverse orientation from~\eqref{15}, only
	their squared norms enter Lemma~\ref{lem:apriori}; hence~\eqref{14}--\eqref{15}.

	For~\eqref{12}, compare $(x,y)$ with the zero pair, which solves the
	\(\alpha=1\) source equation with
\begin{flalign*}
&\xi^0=-\Lambda(0),\qquad \eta^0=-\Phi(0,0),&\\
&\varphi^0_k=-f(k+1,0,0),&\\
&\psi^0_k=-b(k,0,0),\qquad \gamma^0_k=-\sigma(k,0,0).&
\end{flalign*}
These sources belong to $\mathcal H(\overline{\mathbb T})$ by Assumption~\ref{ass:lipschitz}. Lemma~\ref{lem:apriori} then yields~\eqref{12}--\eqref{13}.

Suppose now that the second monotonicity alternative holds. Apply the sign transformation in Remark~\ref{rem:sign} to the reference coefficients and to both solution pairs. The transformed reference coefficients satisfy the first alternative, so the result just proved applies; the comparison pair is still an admissible solution with square-integrable evaluated differences. Since the transformation $(x,y)\mapsto(x,-y)$ is an isometry of $N^2(\overline{\mathbb T};\mathbb R^{2n})$ and leaves all squared source differences unchanged, existence, uniqueness, and estimates~\eqref{12} and~\eqref{14} transfer back to the original system.
\end{proof}

\section{Linear--Quadratic Buffer-Adjustment Illustration}\label{sec:application}
\subsection{Control problem and result}
Consider a one-dimensional finite-horizon buffer-adjustment illustration
motivated by aggregate interbank effects~\cite{CarmonaFouqueSun2015}.
The state $x_k$ is positive for surplus and negative for shortage; $u_k$ is
the adjustment chosen before the next innovation. The term $\mathbb E[u_k]$
describes unconditional average adjustment; it is deterministic and is not a
common-noise conditional aggregate (Section~\ref{sec:scope}). The dynamics are
\begin{flalign}\label{eq:lq-state}
&\begin{aligned}
x_{k+1}={}&a_kx_k+\alpha_ku_k+\bar\alpha_k\mathbb E[u_k]+d_k\\
&+\bigl(c_ku_k+\bar c_k\mathbb E[u_k]+\delta_k\bigr)w_k,
\qquad x_0=x,
\end{aligned}&
\end{flalign}
for $k\in\mathbb T$, and the cost functional
\begin{flalign}\label{eq:lq-cost}
&\begin{aligned}
J(u)=\frac12\mathbb E\Bigg[&\sum_{k=0}^{N-1}\Big(q_k|x_k^{(1)}|^2
+\bar q_k|x_k^{(2)}-m_k|^2\\
&+r_k|u_k^{(1)}|^2+(r_k+\bar r_k)|u_k^{(2)}|^2+2\ell_kx_k\Big)\\
&+h|x_N^{(1)}|^2+\bar h|x_N^{(2)}-m_N|^2+2\ell_Nx_N\Bigg],
\end{aligned}&
\end{flalign}
where $v^{(1)}=v-\mathbb E[v]$ and $v^{(2)}=\mathbb E[v]$. After taking expectations, the two control terms equal
$r_k\mathbb E|u_k|^2+\bar r_k|\mathbb E[u_k]|^2$.
The state cost penalizes variance and deviation of the population mean from
$m_k$; it does not impose an asymmetric individual shortage penalty, and the
model is a stylized convex adjustment problem rather than a systemic-risk
formulation (Section~\ref{sec:scope}).
Here $a_k$ describes buffer persistence, $(\alpha_k,\bar\alpha_k)$ measure the direct and aggregate effects of adjustment on the conditional mean, and $(c_k,\bar c_k)$ measure the corresponding innovation exposure. We specialize to $n=m=1$, so the admissible set $\mathbb U$ is understood to contain scalar controls.

\begin{assumption}\label{ass:lq}
The coefficient sequences $a_k,\alpha_k,\bar\alpha_k,c_k,\bar c_k$,
$q_k,\bar q_k,r_k,\bar r_k$ and the terminal weights $h,\bar h$ are deterministic and uniformly bounded. There exists $\lambda>0$ such that
\begin{flalign*}
	&
	r_k\ge\lambda,
	\qquad r_k+\bar r_k\ge\lambda,
	\qquad k\in\mathbb T.
	&
\end{flalign*}
Moreover, $q_k,\bar q_k,h,\bar h\ge0$; the processes $d(\cdot)$ and $\delta(\cdot)$ belong to $L^2_{\mathbb F}(\mathbb T;\mathbb R)$ and are therefore $\mathcal F_{k-1}$-measurable at time $k$; and $m_k,\ell_k$ ($k\in\overline{\mathbb T}$) are deterministic and bounded. The coefficient \(\alpha_k\) is unrelated to the continuation parameter \(\alpha\) used in Section~\ref{sec:proofs}.
\end{assumption}

\begin{problem}\label{prob:lq}
Find $u^*\in\mathbb U$ such that $J(u^*)=\inf_{u\in\mathbb U}J(u)$ subject to~\eqref{eq:lq-state}.
\end{problem}

This positive-weight scalar problem illustrates the nonlinear theorem's
Hamiltonian mechanism within an established MF-LQ class, and its optimizer also
exists by coercive quadratic optimization on the control Hilbert space. The aim
here is to verify the abstract assumptions of Sections~\ref{sec:formulation}
and~\ref{sec:proofs} directly on a concrete MF-LQ instance; the resulting scope,
in particular the absence of a Riccati-based feedback representation, is stated
in Section~\ref{sec:scope}.

\begin{theorem}\label{thm:lq}
Assumption~\ref{ass:lq} permits a choice $\mu>0$ and $\nu=0$ for which the
coefficients in~\eqref{eq:lq-coefficients} meet
Assumptions~\ref{ass:lipschitz} and~\ref{ass:dm} in the first monotonicity
orientation; some domination matrices may be degenerate. The associated
Hamiltonian system~\eqref{11} is therefore well posed, with its sole solution
\begin{flalign*}
	&
	(x^*(\cdot),y(\cdot))\in
	N^2(\overline{\mathbb T};\mathbb R^2).
	&
\end{flalign*}
	The control in~\eqref{eq:optimal-control}, evaluated at this Hamiltonian
solution, is the unique open-loop optimizer of Problem~\ref{prob:lq}.
\end{theorem}

\subsection{Derivation of the Hamiltonian system}
The derivation is elementary but lengthy: it differentiates the affine state
recursion, introduces the adjoint process, and evaluates the resulting
first-order condition. The complete computation, which yields the
stationarity condition~\eqref{eq:stationarity}, the open-loop
control~\eqref{eq:optimal-control}, and the Hamiltonian
coefficients~\eqref{eq:lq-coefficients}, is collected in
Appendix~\ref{app:hamiltonian}.

\subsection{Structural verification and optimality proof}
\begin{proof}[Proof of Theorem~\ref{thm:lq}]
\emph{Step 1: regularity and domination.}
The coefficients~\eqref{eq:lq-coefficients} are affine in $(\theta,\zeta)$ with
bounded deterministic linear parts. Their values at zero are
\[
\Lambda(0)=x,\quad\Phi(0,0)=-\bar h m_N+\ell_N,
\]
and $\Gamma(k,0,0)=(\bar q_km_k-\ell_k,d_k,\delta_k)$.
Assumption~\ref{ass:lipschitz} follows from Assumption~\ref{ass:lq}.
For Case~A take $M=G=\widetilde G=A_k=\widetilde A_k=0$ and use the
displayed $B_k,C_k,\widetilde B_k,\widetilde C_k$ as domination matrices.
Let
\[
K_0=\sup_k\max\{|B_k|,|C_k|,|\widetilde B_k|,|\widetilde C_k|\}.
\]
Choose
\begin{align*}
0<\mu&\le\min\bigl\{1,(1+K_0)^{-1},\\
&\qquad(1+\kappa_B)^{-1/2},(1+\kappa_C)^{-1/2}\bigr\},\qquad\nu=0.
\end{align*}
This gives $\mu K_0\le1$ and the normalization in
Assumption~\ref{ass:dm}(iv). Initial domination is immediate.
For differences with $X$ fixed, put
\[
V_1=B_k\widehat Y^{(1)}+C_k\widehat Z^{(1)},\qquad
V_2=\widetilde B_k\widehat Y^{(2)}+\widetilde C_k\widehat Z^{(2)}.
\]
Then $\Delta b=-B_kV_1-\widetilde B_kV_2$ and
$\Delta\sigma=-C_kV_1-\widetilde C_kV_2$.
Because $V_1$ is centered and $V_2$ deterministic,
\begin{align*}
\mathbb E|\Delta b|^2
&=B_k^2\mathbb E|V_1|^2+\widetilde B_k^2|V_2|^2\\
&\le\mu^{-2}\mathbb E|V_1+V_2|^2,
\end{align*}
and the identical estimate with $C_k,\widetilde C_k$ proves diffusion domination.
Zero or rank-deficient combinations cause no difficulty: their kernels also
annihilate these coefficient differences.

\emph{Step 2: monotonicity.}
The initial pairing is zero and the terminal pairing is
\[
\mathbb E[\Delta\Phi\widehat X]
=h\mathbb E|\widehat X^{(1)}|^2+\bar h|\widehat X^{(2)}|^2\ge0.
\]
For arbitrary differences, the $a_k$ terms cancel, and orthogonality gives
\begin{align*}
\mathbb E\langle\Delta\Gamma,\widehat\Theta\rangle
={}&-q_k\mathbb E|\widehat X^{(1)}|^2-\bar q_k|\widehat X^{(2)}|^2\\
&-\mathbb E|V_1|^2-|V_2|^2\\
\le{}&-\mu\mathbb E|V_1+V_2|^2.
\end{align*}
Thus Assumption~\ref{ass:dm} holds. Theorem~\ref{thm:main} supplies the unique
Hamiltonian solution $(x^*,y)\in N^2$.

\emph{Step 3: admissibility, optimality, and uniqueness.}
Bounded coefficients and $r_k,r_k+\bar r_k\ge\lambda$ give
\[
\mathbb E|u_k^*|^2\le C\mathbb E(|y'_{k+1}|^2+|z'_{k+1}|^2)
\le2C\mathbb E|y_{k+1}|^2.
\]
The last inequality is~\eqref{eq:projection-bound}; predictability and summation
give $u^*\in\mathbb U$. The state map from $\mathbb U$ to the state $L^2$
space is bounded and affine by the forward estimate. Hence $J$ is a continuous
quadratic functional with a Fr\'echet derivative, so the variation and expansion
used below are valid on the entire control Hilbert space. Substitution in~\eqref{eq:lq-state} recovers $x^*$.
For any $u\in\mathbb U$, set $\Delta u=u-u^*$ and $\Delta x=x-x^*$.
Affinity of the state equation makes $\Delta x$ its exact first variation.
The exact quadratic expansion is
\begin{align*}
J(u)-J(u^*)={}&DJ(u^*)\Delta u+\tfrac12\mathbb E\Bigg[
\sum_{k=0}^{N-1}\Big(q_k|\Delta x_k^{(1)}|^2\\
&+\bar q_k|\Delta x_k^{(2)}|^2+r_k|\Delta u_k^{(1)}|^2\\
&+(r_k+\bar r_k)|\Delta u_k^{(2)}|^2\Big)\\
&+h|\Delta x_N^{(1)}|^2+\bar h|\Delta x_N^{(2)}|^2\Bigg].
\end{align*}
Duality~\eqref{eq:duality-lq} and stationarity~\eqref{eq:stationarity} make the
first term zero. All state weights are nonnegative, so
\[
J(u)-J(u^*)\ge\frac\lambda2\mathbb E\sum_{k=0}^{N-1}|u_k-u_k^*|^2.
\]
This proves optimality and uniqueness in the control Hilbert space.
\end{proof}

\subsection{Exact finite scenario-tree computation}\label{subsec:computation}
For a reproducible illustration take $N=4$, $x_0=-0.5$, and independent
Rademacher innovations, each equal to $-1$ or $1$ with probability $1/2$.
Use the following constants at every running time:
\begin{align*}
&(a,\alpha,c,d,\delta)=(0.9,0.6,0.2,-0.1,0.25),\\
&(q,\bar q,r,\bar r)=(0.5,1,0.4,0.2),\qquad(h,\bar h)=(1,2),\\
&m_k=\ell_k=0\quad(0\le k\le N),\\
&\bar\alpha=\eta,\qquad\bar c=0.1\eta,
\qquad\eta\in\{0,0.3,0.6\}.
\end{align*}
Here $\eta$ is an experiment parameter, unrelated to the terminal source in
the continuation proof. All three configurations satisfy Assumption~\ref{ass:lq}.
This is a synthetic model check, not an empirical calibration.

At time $k$ there are $2^k$ equiprobable histories. Assign one control to each
history, so the entire admissible control is a vector $v\in\mathbb R^{15}$.
With children ordered as $(-1,+1)$, exact conditional projections are
\[
y'_{k+1}=\tfrac12(y^-_{k+1}+y^+_{k+1}),\qquad
z'_{k+1}=\tfrac12(y^+_{k+1}-y^-_{k+1}).
\]
Unconditional means are averages over the $2^k$ histories at each level.
Because the state is affine in $v$, the cost is
\[
J(v)=\tfrac12v^\top H v+g^\top v+J(0).
\]
The direct method constructs $H$ and $g$ from the affine state sensitivities
and solves $Hv=-g$. A separate Hamiltonian method constructs the terminal and
backward adjoints from the states, forms the nodewise stationarity residual
$s(v)=Kv+s(0)$ using~\eqref{eq:stationarity}, and solves $Kv=-s(0)$.
The implementations additionally check $H=WK$, where $W$ is diagonal with
entry $2^{-k}$ for each time-$k$ control. This accounts for the probability
weights in the Euclidean cost gradient. Both methods enumerate the entire tree;
there is no sampling or Monte Carlo error.

\begin{table}[t]
\centering\small
\caption{Exact scenario-tree results, rounded to six decimals.}\label{tab:tree-results}
\begin{tabular}{rrrrr}
\hline
$\eta$ & $u_0^*$ & $\mathbb E[x_N^*]$ & $\mathrm{sd}(x_N^*)$ & $J(u^*)$\\
\hline
0.0 & 0.504422 & -0.111067 & 0.328640 & 0.485773\\
0.3 & 0.418873 & -0.062637 & 0.319500 & 0.390460\\
0.6 & 0.352644 & -0.044631 & 0.314766 & 0.347350\\
\hline
\end{tabular}
\end{table}

The uncontrolled cost is $J(0)=1.282822$ in all three cases. This common value
is expected: $\eta$ enters the dynamics only through the aggregate terms
$\bar\alpha_k\mathbb E[u_k]$ and $\bar c_k\mathbb E[u_k]w_k$, both of which
vanish when $u\equiv0$, so the parameter drops out of the state recursion and
hence out of the cost. The control energy
$\mathbb E\sum_k|u_k^*|^2$ is respectively $0.573689$, $0.397404$, and
$0.306183$. Across the three configurations, the largest difference between
the two computed control vectors is below $8\times10^{-16}$; the largest
nodewise stationarity residual is below $1.3\times10^{-15}$; and
$\|H-WK\|_{\max}<3.6\times10^{-15}$. The smallest Euclidean Hessian eigenvalue
exceeds $0.0759$ in every case. These are floating-point checks of the finite
instance, complementing the analytical uniqueness proof.

The same two constructions were repeated at larger horizons to check that the
agreement is not an artifact of the smallest tree. For $N=6$ ($63$ controls) the
largest discrepancy between the two computed control vectors is below
$5.6\times10^{-15}$ and the largest nodewise stationarity residual below
$1.4\times10^{-15}$; for $N=8$ ($255$ controls) the corresponding bounds are
$2.1\times10^{-14}$ and $7.2\times10^{-15}$. The smallest Euclidean Hessian
eigenvalue decreases from $0.0760$ at $N=4$ through $0.0190$ at $N=6$ to
$0.0047$ at $N=8$: the quadratic form remains strictly coercive at every horizon
tested, while the margin shrinks as the tree grows. Both horizons are produced by
the same script described below.

For these selected positive aggregate-response coefficients, increasing $\eta$
reduces the optimal initial adjustment, terminal mean shortage, and total cost.
The parameters $\bar\alpha$ and $\bar c$ vary together, so this comparison
does not isolate the separate drift and diffusion effects or establish a general
monotonicity law. The same finite-tree construction can also compute
Riccati-based policies; the present cross-check instead uses direct quadratic
optimization, independently of the Hamiltonian stationarity assembly.
The accompanying script \texttt{revision0930\_liquidity\_check.py} records
all nodewise states, adjoints, controls, and unrounded diagnostics in JSON and
summary results in CSV, including the $N=6$ and $N=8$ runs reported above.

\section{Scope and Limitations}\label{sec:scope}
The results are deliberately restricted, and the restrictions that remain in
force are collected here.

\emph{Sufficient, coefficient-level conditions.} Assumptions~\ref{ass:lipschitz}
and~\ref{ass:dm} are sufficient conditions. No necessary conditions and no
characterization of all solvable coefficient families are provided. The theory
is stated at the level of the coefficients: only the first moment
$\mathbb E[\theta_k]$ and the terminal mean $\mathbb E[x_N]$ are coupled, and
dependence on the full probability law of the state is not covered.

\emph{No linear characterization.} For linear MF-FBS$\Delta$Es the necessary
and sufficient Riccati criterion of~\cite{Ma2025} is not recovered: the
assumptions above remain sufficient, and they do not locate the boundary
between solvable and non-solvable linear coefficient families.

\emph{Unconditional means.} The mean-field terms involve unconditional
expectations. Scalar innovations satisfying
$\mathbb E[w_k^2\mid\mathcal F_{k-1}]=1$, together with unconditional means, do
not model a population law conditioned on common noise. Conditioning the
aggregate law on a common noise would change the projection structure
$y'_{k+1}=\mathbb E[y_{k+1}\mid\mathcal F_{k-1}]$,
$z'_{k+1}=\mathbb E[y_{k+1}w_k\mid\mathcal F_{k-1}]$ and is not treated.

\emph{One active coercivity direction.} The continuation anchor~\eqref{17}
requires a sequentially solvable system at $\alpha=0$, so a single coercivity
direction is retained. Systems with two genuinely active directions are covered
only through Remark~\ref{rem:unilateral}, that is, after discarding one of them.
The solution constants depend on the active parameter through $\mu^{-2}$ or
$\nu^{-2}$ (see \eqref{eq:mu-dependence}--\eqref{eq:nu-dependence}) and are not
bounded uniformly as it tends to zero; by Lemma~\ref{lem:normalization} this is
not an additional restriction on coefficients already satisfying (i)--(iii).

\emph{The LQ part is an illustration.} The scalar buffer-adjustment model
verifies the abstract assumptions within an established MF-LQ class and exposes
the Hamiltonian mechanism. It does not supply a new general MF-LQ solvability
criterion, a Riccati-based state-feedback representation, or an economic
calibration. Its control formula is evaluated along the Hamiltonian solution,
and the exact scenario-tree computation is a finite-instance check at $N=4$.

\emph{Unaddressed extensions.} Constraints on the state or the control, infinite
horizons, law-dependent coefficients, and scalable numerical methods require
additional analysis and are outside the present results.

\section{Conclusion}\label{sec:conclusion}
For a finite-horizon nonlinear stochastic difference system with unconditional
first-moment interactions, domination of frozen coefficient differences and
coercivity of the same centered/mean matrix combinations close the continuation
estimate. This yields a unique square-integrable adapted solution and coefficient
stability. The bounds are uniform in the homotopy parameter, with dependence on
the active coercivity constant. Normalization is always available for fixed
bounded domination matrices, and stronger two-sided conditions can be weakened
to the one-sided construction. The nonlinear saturation example verifies that
the result extends beyond the linear class and permits degenerate combinations,
and Corollary~\ref{cor:degenerate} shows that switching off the mean-field terms
returns the non-mean-field theory.

The positive-weight scalar LQ buffer model illustrates the theorem within
existing mean-field control theory. Its optimality formula is an open-loop
Hamiltonian characterization, and the exact scenario-tree comparison supplies a
finite-instance computational check. The limitations of the setting are listed
in Section~\ref{sec:scope}.

\appendix
\section{Derivation of the Hamiltonian system}\label{app:hamiltonian}
This appendix supplies the derivation summarized in
Section~\ref{sec:application}. The equation numbers below are those of the
appendix; all labels are shared with the main text.

Fix $u\in\mathbb U$ and a direction $v\in\mathbb U$. Let $x^\varepsilon$ be the state corresponding to $u+\varepsilon v$. Since the state equation is affine, the variational process $\xi_k=(x_k^\varepsilon-x_k)/\varepsilon$ is independent of $\varepsilon$ and satisfies
\begin{flalign}\label{eq:variation}
&\begin{cases}
\begin{aligned}
\xi_{k+1}={}&a_k\xi_k+\alpha_kv_k+\bar\alpha_k\mathbb E[v_k]\\
&+\bigl(c_kv_k+\bar c_k\mathbb E[v_k]\bigr)w_k,
\end{aligned}\\
\xi_0=0.
\end{cases}&
\end{flalign}
Introduce the adjoint process
\begin{flalign}\label{eq:adjoint}
&\begin{cases}
y_k=a_ky'_{k+1}+q_kx_k^{(1)}+\bar q_k(x_k^{(2)}-m_k)+\ell_k,\\
y_N=h x_N^{(1)}+\bar h(x_N^{(2)}-m_N)+\ell_N,
\end{cases}&
\end{flalign}
where $y'_{k+1}=\mathbb E[y_{k+1}\mid\mathcal F_{k-1}]$ and $z'_{k+1}=\mathbb E[y_{k+1}w_k\mid\mathcal F_{k-1}]$. Multiplying~\eqref{eq:variation} by $y_{k+1}$ and taking conditional expectations gives, for each $k$,
\begin{flalign*}
&\begin{aligned}
\mathbb E[\xi_{k+1}y_{k+1}]={}&
\mathbb E\Big[\big(a_k\xi_k+\alpha_kv_k+\bar\alpha_k\mathbb E[v_k]\big)y'_{k+1}\Big]\\
&+\mathbb E\Big[\big(c_kv_k+\bar c_k\mathbb E[v_k]\big)z'_{k+1}\Big].
\end{aligned}&
\end{flalign*}
Summing the adjoint identity and using \(\xi_0=0\) cancels the state-variation
terms, leaving
\begin{flalign}\label{eq:duality-lq}
&\begin{aligned}
DJ(u)v=\mathbb E\sum_{k=0}^{N-1}\Big[&\bigl(r_ku_k^{(1)}+\alpha_ky_{k+1}'^{(1)}+c_kz_{k+1}'^{(1)}\bigr)v_k^{(1)}\\
&+\bigl((r_k+\bar r_k)u_k^{(2)}+(\alpha_k+\bar\alpha_k)y_{k+1}'^{(2)}\\
&+(c_k+\bar c_k)z_{k+1}'^{(2)}\bigr)v_k^{(2)}\Big].
\end{aligned}&
\end{flalign}
The centered and mean components of admissible variations can be varied independently. Indeed, for any square-integrable adapted centered process \(\tilde v_k\) with \(\mathbb E[\tilde v_k]=0\) and any deterministic square-summable sequence \(\bar v_k\), the process \(v_k=\tilde v_k+\bar v_k\) belongs to \(\mathbb U\), with \(v_k^{(1)}=\tilde v_k\) and \(v_k^{(2)}=\bar v_k\). Therefore the first-order condition \(DJ(u^*)v=0\) for every \(v\in\mathbb U\) is equivalent to
\begin{flalign}\label{eq:stationarity}
&\begin{cases}
r_ku_k^{*(1)}+\alpha_ky_{k+1}'^{(1)}+c_kz_{k+1}'^{(1)}=0,\\
\begin{aligned}
&(r_k+\bar r_k)u_k^{*(2)}+(\alpha_k+\bar\alpha_k)y_{k+1}'^{(2)}\\
&\qquad +(c_k+\bar c_k)z_{k+1}'^{(2)}=0.
\end{aligned}
\end{cases}&
\end{flalign}
Solving~\eqref{eq:stationarity},
\begin{flalign}\label{eq:optimal-control}
&\begin{aligned}
u_k^*={}&-\frac{\alpha_ky_{k+1}'^{(1)}+c_kz_{k+1}'^{(1)}}{r_k}\\
&-\frac{(\alpha_k+\bar\alpha_k)y_{k+1}'^{(2)}}{r_k+\bar r_k}
-\frac{(c_k+\bar c_k)z_{k+1}'^{(2)}}{r_k+\bar r_k}.
\end{aligned}&
\end{flalign}

Formula~\eqref{eq:optimal-control} is an implicit Hamiltonian characterization:
the projected adjoint must be obtained with the optimal state. It does not by
itself give a state-feedback implementation. Section~\ref{subsec:computation}
computes the open-loop solution on a finite tree.

Set
\begin{flalign*}
& B_k=\frac{\alpha_k}{\sqrt{r_k}},\qquad
C_k=\frac{c_k}{\sqrt{r_k}}, &\\[-0.2em]
& \widetilde B_k=\frac{\alpha_k+\bar\alpha_k}{\sqrt{r_k+\bar r_k}},\qquad
\widetilde C_k=\frac{c_k+\bar c_k}{\sqrt{r_k+\bar r_k}}. &
\end{flalign*}
Set
\begin{flalign*}
&\mathcal V_k^{(1)}=B_ky_{k+1}'^{(1)}+C_kz_{k+1}'^{(1)},&\\
&\mathcal V_k^{(2)}=\widetilde B_ky_{k+1}'^{(2)}+\widetilde C_kz_{k+1}'^{(2)}.&
\end{flalign*}
Substitution of~\eqref{eq:optimal-control} into~\eqref{eq:lq-state} gives the Hamiltonian MF-FBS$\Delta$E~\eqref{11}, with $\Lambda(y)\equiv x$ and
\begin{flalign}\label{eq:lq-coefficients}
&\left\{
\begin{aligned}
b(k,\theta_k,\mathbb E[\theta_k])={}&a_kx_k+d_k-B_k\mathcal V_k^{(1)}\\
&-\widetilde B_k\mathcal V_k^{(2)},\\
\sigma(k,\theta_k,\mathbb E[\theta_k])={}&\delta_k-C_k\mathcal V_k^{(1)}\\
&-\widetilde C_k\mathcal V_k^{(2)},\\
f(k+1,\theta_k,\mathbb E[\theta_k])={}&-a_ky_{k+1}'-q_kx_k^{(1)}\\
&-\bar q_kx_k^{(2)}+\bar q_km_k-\ell_k,\\
\Phi(x_N,\mathbb E[x_N])={}&h x_N^{(1)}+\bar h x_N^{(2)}\\
&-\bar h m_N+\ell_N.
\end{aligned}
\right.&
\end{flalign}

\section*{CRediT authorship contribution statement}
\textbf{Yingjie Zheng:} Conceptualization, Methodology, Formal analysis,
Software, Validation, Writing -- original draft.
\textbf{Qingxin Meng:} Conceptualization, Methodology, Supervision, Writing --
review and editing.
\textbf{Maoning Tang:} Conceptualization, Supervision, Project administration,
Writing -- review and editing.

\section*{Declaration of competing interest}
The authors declare that they have no known competing financial interests or
personal relationships that could have appeared to influence the work reported
in this paper.

\section*{Data availability}
No external data were used. All numerical results are reproduced by the
accompanying script \texttt{revision0930\_liquidity\_}\allowbreak\texttt{check.py},
which is provided with the submission and generates both the summary table and
the nodewise diagnostics.

\section*{Acknowledgments}
This work was supported by the National Natural Science Foundation of China
(No.~12271158), the Key Program of Zhejiang Provincial Natural Science
Foundation (No.~LZ22A010005), the Natural Science Foundation of Huzhou City
(No.~2023YZ46), and the Postgraduate Research and Innovation Project of Huzhou
Normal University (No.~2024KYCX66).

\end{document}